\documentclass{elsarticle}%
\usepackage{amssymb}
\usepackage{amsfonts}
\usepackage{amsmath,amsthm}
\usepackage{graphicx,fullpage,xcolor}
\usepackage{subcaption,enumerate}
\newtheorem{theorem}{Theorem}
\theoremstyle{plain}

\newtheorem{definition}{Definition}

\newtheorem{proposition}[definition]{Proposition}
\newtheorem{remark}[definition]{Remark}

\newtheorem{corollary}[definition]{Corollary}

\numberwithin{equation}{section}

\newcommand{\R}{\mathbb{R}}
\newcommand{\C}{\mathbb{C}}
\newcommand{\RR}{\mathcal{R}}
\renewcommand{\SS}{\mathcal{S}}
\newcommand{\EE}{\mathcal{E}}
\newcommand{\CC}{\mathcal{C}}
\newcommand{\LL}{\mathfrak{L}}
\newcommand{\Centr}{\mathfrak{C}}
\newcommand{\KK}{\mathcal{K}}
\newcommand{\Id}{\mathrm{Id}}
\newcommand{\Exp}[1]{\mathrm{Exp}(#1)}

\begin{document}
\title{Jointly Equivariant Dynamics for Interacting Particles}
\author{Alain Ajami}
\address{Universit\'{e} Saint Joseph de Beyrouth, Facult\'{e} d'ing\'{e}gnierie, ESIB, Lebanon, {\tt alain.ajami@usj.edu.lb}}
\author{Jean-Paul Gauthier}
\address{Universit\'{e} de Toulon, LIS, UMR CNRS 7020, Campus de La Garde, 83041
Toulon, France {\tt gauthier@univ-tln.fr}}
\author{Francesco Rossi\footnote{Corresponding Author. The Author is a member of G.N.A.M.P.A. (I.N.d.A.M.).}}
\address{Dipartimento di Culture del Progetto, Università Iuav di Venezia, Italy {\tt francesco.rossi@iuav.it}}
\date{\today}

\begin{abstract}Let a finite set of interacting particles be given, together with a symmetry Lie group $G$. Here we describe all possible
dynamics that are jointly equivariant with respect to the action of $G$. This
is relevant e.g., when one aims to describe collective dynamics that
are independent of any coordinate change or external influence.

We particularize the results to some key examples, i.e. for the most
basic low dimensional symmetries that appear in collective dynamics on manifolds.

\end{abstract}

\begin{keyword} Large-scale systems, Lie groups, equivariant dynamics, systems with symmetry
\end{keyword}

\maketitle


\section{Introduction}
The aim of this article is to characterize all possible interaction models for
$N$ agents evolving on a manifold $M $, that are jointly equivariant with
respect to the action of a group $G$ on $M$.  Equivariance means that an
identical action of the group on each agent provides a change of the
interaction vector field that is equivariant with respect to such action. As a
result, the corresponding dynamics is equivariant too, hence the relative
dynamics between agents is unchanged under the group action.

The simplest (and very interesting) example is the case in which a group of agents evolve on the plane
$M=\mathbb{R}^{2}$ and $G=SE(2)$ is the Lie group of rototranslations: joint
equivariance of interactions in this case means that a rototranslation of the
initial data induces the same rototranslation of the trajectory. More deeply,
this implies that one of the agents (or an external observer subject to the same
rototranslation) will perceive the same trajectory. This example shows the
relevance of the research here: to classify all interactions based on some
\textquotedblleft natural\textquotedblright\ features that can be perceived by
agents and not on some \textquotedblleft superimposed\textquotedblright\  
 structures (such as coordinates), that are added to the model just for other 
purposes (e.g. for parametrization).

The case of dynamics that are invariant with respect to rototranslations, that
we will describe in full details in Section \ref{s-SE2}, has been studied in a
cornucopia of articles and examples. A nonexhaustive bibliography is the
following: \cite{Un, Trois, Quatre, Cinq,Six, Huit, Dix, Douze, CS,
Seize, Dix9,Vingt,Vingt1,Vingt3}.  The main result of our article in this
setting is not to propose yet other models, but to prove a much stronger
result: we will provide all possible dynamics in full generality. Moreover,
such general dynamics will be written in some normal forms, in which the role
of each term is identified for the purpose of modeling.

This example also shows the interest of the classification from a reverse
point of view. If we describe the dynamics of a group of agents with a given
interaction model and remark that it is not jointly equivariant with respect
to a natural group action (e.g. the rototranslations on the plane as above),
this means that an external factor breaks down the natural symmetry of the
group. As an example, several models describing flocks flying on large
distances on the Earth are not jointly equivariant with respect to the natural
group of rotations on the sphere, since they take into account the geomagnetic
field. If the description of the dynamics based on these models is efficient,
this can be seen as a (secondary) evidence of the fact that birds perceive the
magnetic field, i.e. their magnetoreceptivity. See e.g. \cite{Vingt4,Vingt6}.\\


The main aim of this article is then twofold. On a more theoretical side, we
aim to find all possible interaction models, under some assumptions of
equivariance or joint equivariance. More precisely, given a manifold $M$ and a
Lie group $G$ acting on it, we want to describe all dynamics of $N$ particles
evolving on $M$ that are jointly equivariant with respect to the action of
$G$. In particular, we want to identify the minimal number of (functional)
parameters that characterize such dynamics. On the more applied side, we aim
to understand the role of such parameters in modeling, i.e. which
\textquotedblleft part\textquotedblright\ of interactions they describe. With
this goal, we then find all interaction models for key examples determined by triples $(N,M,G)$. We finally choose simple functional
parameters (e.g. constant or linear functions) to understand their role via
some simulations.

One of the interesting features of this study is that the results depend on
the number of agents $N$ in an unexpected way. On one side, the case $N = 1$ is often
far from being trivial (see for instance Sections \ref{s-SO2} and \ref{s-SL2}). On the other
side, some results are proved when $N$ is sufficiently large with respect to the
dimension of the manifold (see Theorem \ref{mainth}). As a consequence, our
methods do not cover \textquotedblleft intermediate\textquotedblright\ numbers
of agents, in general, (see e.g. Remark \ref{r-proper-inclusion-2}).

All along the article, we consider interacting agents that may have
distinct features, i.e. with distinct rules for dynamics. Also, in several examples (starting from statistical mechanics), agents are supposed
to satisfy identical rules for the dynamics. In other terms, agents are only
distinguishable by their state (e.g. position, or position/velocity) and not by other features. We
call such models ``permutation-equivariant'' populations. From the mathematical
viewpoint, this means that the dynamics is equivariant with respect to
permutation of agents: if $x_{i}(0)$ and $x_{j}(0)$ are exchanged, their
trajectories are exchanged too, while trajectories of other agents do not
vary. We then study the family of interaction models that satisfy
equivariance with respect to both the action of a connected Lie group and
to permutations of agents. Clearly, these families have much less parameters
than the ones studied before, but their structure is very similar.

As already stated, in this article we deal with a finite number $N$ of
interacting agents. The limit $N\rightarrow+\infty$ of permutation-equivariant
populations, mathematically defined via the mean-field limit, has been hugely
studied in the literature, see e.g.\cite{Neuf,Treize,elam,Dix7,Dix8,Vingt2}. This
setting plays a key role as a good approximation of the dynamics when the
number of agents is very large, e.g. in bird flocks. We aim to describe the
class of admissible equivariant dynamics in the mean-field limit in a future work.

We also consider the case of controlled dynamics, in which the dynamics $F$
is constrained to belong to a subset of the vector fields on $M$. The most
interesting case here is given by nonholonomic systems, in which a constraint on
the velocity of agents is added to the interaction field. As an example, the
unicycle describes the dynamics of an oriented agent, that can move forward or
rotate on itself: this constrained dynamics is also used for modeling
pedestrians, see e.g. \cite{Deux, Onze}.

All along the article, we deal with first-order dynamics (eventually with
nonholonomic constraints as explained above). This choice, even though very common in many
interaction models (e.g. \cite{Cinq}), is in contrast with a large number of
other models, in which dynamics are of second-order
\cite{Un,Trois,Sept,Vingt5}. The second-order models are chosen for several
different reasons:

\begin{itemize}
\item when velocities play the role of \textquotedblleft state
variables\textquotedblright\ for the dynamics, e.g. when one aims to describe
alignment in bird flocks (e.g. in the celebrated Cucker-Smale model
\cite{CS});

\item when the basic dynamics of agents is of second-order, e.g. when dealing
with classical equations of physics;

\item when one aims to describe orientation of agents via their velocity vectors.
\end{itemize}

In this last case, we show that a different solution can be introduced,
via constrained dynamics. As we already recalled above, one can describe a
dynamics in which the trajectory is always tangent to the orientation of the
body, by adding an angle variable and a nonholonomic constraint. This
description has the advantage of directly describing the desired constraint,
without resorting to higher-order dynamics that are not justified by the
physical model. See more details in Section \ref{s-unicycle}.\\

There is an enormous amount of mathematical contributions about equivariant dynamical systems, which theory is well-developed. Among many of them, the following references are particularly interesting for our setting: \cite{field,kla,matsui}. In the present work, we do not use the tools and results from this field. The originality of our contribution is the choice of the specific diagonal action for several agents. In a forthcoming paper, we will present (still in the context of the diagonal action) several results about equilibria, relative equilibria and stabilization. The present paper also aims to highlight the mathematical interest of equivariant diagonal actions.\\

The structure of the article is the following. We give the main definition related to equivariance in Section \ref{s-general}. There, we also prove here the main general theorems about the structure of Families of Population Dynamics and their properties. We then study several key examples, in which we also highlight the usefulness of the theory we developed:
\begin{itemize}
\item In Section \ref{s-SE2} we study the most important case: the dynamics on the plane $\R^2$ that are equivariant under rototranslations;
\item In Section \ref{relati} we study the relativistic dynamics, both on the line and on the plane;
\item In Section \ref{s-sphere} we study the dynamics on both the circle $S^1$ and the sphere $S^2$, where the group actions are the natural ones ($SO(2,\R)$ and $SO(3, \R)$);
\item In Section \ref{s-SL2}, we briefly present the volume-preserving action of $SL(2)$ on the plane;
\item In Section \ref{s-unicycle}, we describe equivariant dynamics for a well-known non-holonomic control system, namely the unicycle. We present both the classical and relativistic cases.
\item In Section \ref{s-quantum}, we briefly introduce the generalization of our theory to quantum systems and apply it to a system of two quantum agents.
\end{itemize}

In this article, the theory of Lie groups and Lie algebras plays a crucial role. We recall some relevant definitions and results in Appendix.

\section{Definitions and General Results}\label{s-general}

In this section we define the main objects under study, and provide a few
general results about Population Dynamics.

{\bf Notation:} The space of smooth vector fields defined on a domain $D$ is denoted by $\chi(D)$. The Lie derivative of a form $\omega$ along a vector field $v$ will be denoted by $L_{v}\omega$. The matrix multiplication $A.B$ is denoted with a dot.

\subsection{Equivariance: $G$-spaces and radial functions}\label{s-radial}

In this section, we define some objects related to the concept of equivariance. They are stated in terms of Lie groups and Lie algebras, which standard definitions are recalled in Appendix. Throughout the article, we need to work in the analytic setting to ensure equivalence between integral definitions (in terms of Lie groups) and their differential counterpart (in terms of Lie algebras), as it will be clear in the next results.

\begin{definition}[$G$-space] Given $G$ a connected Lie group, an analytic manifold $X$ is a $G$-space if $G$ acts analytically on $X$, with an action denoted by $\Phi(g,x)$.

Let $F$ be an analytic vector field over $X$ and $\mathrm{Exp}(t F(x_0))$ be the associated flow, i.e. the unique solution of the Cauchy problem $$\begin{cases}
\dot x=F(x),\\
x(0)=x_0,
\end{cases}$$
for times $t\in\R$ for which the solution exists. 

The vector field $F$ is said to be $G$-equivariant
if it holds 
\begin{equation}
\Phi(g,\mathrm{Exp}(tF(x)))=\mathrm{Exp}(tF(\Phi(g,x))), \label{equiv1}%
\end{equation}
for all $t,g,x$ such that both sides are defined.
\end{definition}

We now prove the equivalent differential version of \eqref{equiv1}, in the setting of Lie algebras. Additionally, this will prove that the sets on which the left and right hand sides of \eqref{equiv1} are defined coincide. Given $l\in\LL$, the Lie algebra of $G$, we identify it with the vector field $l(x)$ on $X$ as the following: $l(x)=\frac{d}{dt}_{|t=0}\Phi(\mathrm{Exp}(t l),x)$.

\begin{proposition}
\label{equivp1}The vector field $F$ is $G$-equivariant if and only if $[F,l(x)]=0$ for
all $l\in\LL$, $x\in X$.
\end{proposition}
\begin{proof} For simplicity of exposition, we assume that $G$ is an exponential connected linear Lie group, $X$ is an open subset of $\mathbb{R}^{k} $, invariant under the linear $G$-action, i.e. that $\Phi(g,x)=g.x$. In that case, the Lie algebra $\LL$ of $G$ is a Lie subalgebra of the space $\mathcal{M}_{k}$ of $k\times k$ real matrices. The equality \eqref{equiv1} then rewrites as 
\begin{equation}
g.\mathrm{Exp}(tF(x))=\Exp{tF(g.x)}. \label{equiv2}%
\end{equation}

Necessity: Differentiate \eqref{equiv2} at $t=0$ to get $g.F(x)=F(g.x)$. Take  $g=\Exp{vl}$ and differentiate with respect to $v$ at $v=0$ to get the result.

Sufficiency: Assume $[F,l.x]=0 $ for all $l\in\LL$ and define%
\[
A(t,x):=\Exp{-tl}. F(\Exp{tl}.x)=\sum_{r=0}^{+\infty}(ad(l)^{r}F)(x)\frac{t^{r}}{r!},
\]
where $ad(l)^r F:=[l,ad(l)^{r-1} F]$ for $r\in \mathbb{N}\setminus\{0\}$ and $ad(l)^0 F=F$. Here, the series is absolutely convergent in $t$, due to analyticity. Then $A(t,x)=F(x)$, thus $F(\Exp{tl}.x)=\Exp{tl}.F(x)$ for all $x,t$.  It means that $F(g.x)=g.F(x) $, or equivalently 
\begin{equation}
\label{e-F}
F(x)=g^{-1}.F(g.x)\mbox{~~for all~~}g\in G.
\end{equation}

For a fixed $g\in G$, set $B(t,x)=g^{-1}.\Exp{tF(g.x)}$. We then have:%
\begin{equation}\label{e-B}
\partial_tB   =g^{-1}.F(\Exp{tF(g.x)})  =g^{-1}.F(g.B)=F(B),
\end{equation}
where the last identity comes from \eqref{e-F}. By observing that $B(0,x)=x$, we have that \eqref{e-B} implies $B=\Exp{tF(x)}$, that in turn proves \eqref{equiv2}, for small $t$. Moreover, this also shows that both sides of the equality extend to the same $t$.
\end{proof}

We now define radial functions and characterize them in differential terms.
\begin{definition} An analytic function $\phi:D_\phi\subset X \to \R$, with $D_\phi$ being an open dense $G$-space, is radial if it is constant on the $G$-orbits in $X$.
\end{definition}

\begin{proposition}
\label{proprad}The function $\varphi$ is a radial function if and only if \[
L_{lx}\varphi=0\mbox{~~for all~}l\in\LL.
\]
The space of radial functions is an Abelian ring, that we denote by $\RR$.
\end{proposition}
\begin{proof} For $l\in\LL$ and for all $t,x$, it holds $\varphi(\Exp{t.lx})=\sum
_{r=0}^{+\infty}L_{lx}^{r}(\varphi)(x)\frac{t^{r}}{r!}$. Independence of both sides with respect to $t$ is equivalent to have $L_{lx}^{r}(\varphi)(x)=0$ for all $x\in D_\phi$ and $r\geq 1$, that is in turn equivalent to have $L_{lx}(\varphi)(x)=0$ for all $x\in D_\phi$.

It is easy to prove that $\phi,\psi$ radial implies both $\phi+\psi$ radial, by linearity of the derivative, and $\phi\psi$ radial, by the Leibnitz rule. Commutativity is standard.
\end{proof}

\subsection{Population Dynamics} We now focus on Population Dynamics, i.e. on dynamics of a set of interacting agents all moving on the same manifold $M$. All along the article, we will denote by $n$ the dimension of such manifold, while $N\in\mathbb{N}\backslash\{0\}$ will be
the number of agents in the population. The space $X$ on which the dynamics take place will then be an open dense analytic submanifold of $M^N$.

We first define the diagonal action, both of a Lie group and of its Lie algebra, as well as diagonal $G$-spaces.

\begin{definition}
\label{d-diagonal} Let $G$ be a connected Lie group acting on a manifold $M$
with action $\Phi$, and $\LL$ its Lie algebra. Given a number of
agents $N\in\mathbb{N}^{\ast}$, we define:

\begin{itemize}
\item the diagonal action of $G$ as the action of $G$ on $M^{N}$ given by
$$\Phi(g)(X_{1},\ldots,X_{N}):=(\Phi(g)X_{1},\ldots,\Phi(g)X_{N});$$

\item the diagonal extension of the infinitesimal action of $\LL$ as
follows: any $l\in\LL$ defines a vector field $l(X)\frac{\partial
}{\partial X}$ on $M $, and we consider the corresponding vector field on
$M^{N}$ given by $\hat{l}=l_{1}+\ldots+l_{n}$ with $l_{i}=l(X_{i}%
)\partial_{ X_{i}}$. The Lie algebra of such vector fields is
denoted by $\hat{\LL}$;
\item a diagonal $G$-space as a $G$-space  $D\subset M^{N}$ being invariant under the
diagonal action of $G$;

\item the ring of $\hat{\LL}$-jointly radial functions (eventually
$\hat{\LL}_{N}$-jointly radial functions if $N$ needs to be specified) is the Abelian ring of analytic functions $\varphi:D_\varphi\subset
M^{N}\rightarrow\mathbb{R}$, with open dense diagonal $G$-space $D_\varphi$ that are radial with respect to the diagonal action of $G$. Equivalently, by Proposition \ref{proprad}, jointly radial functions are solutions of
\begin{equation}
L_{\widehat{l}}\varphi=0\mbox{~~for all~~} \hat{l}\in\hat\LL.\label{e-jointlyradial}
\end{equation}

\end{itemize}
\end{definition}

We now define Population Dynamics (PD) and the corresponding Family of
Population Dynamics (FPD).

\begin{definition}
\label{d-PD}A \textit{Population Dynamics (PD for short)} is defined by the
elements $(N,M,G,D,F) $, where:

\begin{itemize}
\item $N\in\mathbb{N}\backslash\{0\}$ is the number of agents in the population;

\item $M$ is a differentiable manifold, and each agent $X_i$ with $i=1,\ldots,N$ evolves on it;

\item $G$ is a connected Lie group, acting on $M$ via an action $\Phi$;

\item $D\subset M^{N}$ is a diagonal $G$-space;

\item $F$ is an analytic vector field over $D$, that is $G$-equivariant with respect to the diagonal action of $G$, i.e. each
component $F_{i}$ of  the vector field $F$ (acting on the $i$-th agent)
satisfies
\begin{equation}
T\Phi(g)F_{i}(x_{1},\ldots,x_{N})=F_{i}(\Phi(g)x_{1},\ldots,\Phi(g)x_{N})\text{ for
all}\in G. \label{equivdef}%
\end{equation}

\item the domain of $F$ is maximal, i.e. there exists no vector field $\hat
{F}$ with a domain $\hat{D}\supset D$ being a diagonal $G$-space such that $F(x)=\hat{F}(x)$ for all $x\in
D$.
\end{itemize}

Given $(N,M,G)$ as above, the corresponding \textit{Family of Population
Dynamics} (FPD for short) is the largest family $\mathbb{F}$ of smooth vector
fields such that for all $F\in\mathbb{F}$ it holds:

\begin{itemize}
\item the domain of $F$ is a diagonal $G$-space $D_{F}\subset M^{N}$;

\item $(N,M,G,D_F,F)$ is a population dynamics.
\end{itemize}
\end{definition}

The main goal of the article is to characterize FPD's: given $(N,M,G)$ fixed,
we aim to compute and describe all the possible associated Population Dynamics. We start by characterizing them in terms of Lie algebras. This provides a Lie bracket condition for a vector field $F$ to be a
PD.

\begin{corollary}
\label{lemdef}Let $(N,M,G,D)$ be given as in Definition \ref{d-PD}, with $D$
connected. Let $F$ be a smooth vector field in ${\LARGE \chi}(D)$. Let
$\LL$ be the Lie algebra of $G$ and $\hat{\LL}$ its
diagonal extension. Then, $F$ is jointly equivariant with
respect to the action of $G$ if and only if%
\begin{equation}
[ F,\hat l]=0,\text{ for all }\hat l\in\hat{\LL}. \label{maineq}%
\end{equation}

\begin{remark} Even though we will mainly
deal with a whole domain $D=M^{N}$, the case of $D$ smaller is important too. The most relevant example is given by interaction based on potentials with singularities to prevent particle collisions. In this case, it holds $D=M^{N}\setminus\cup_{i\neq j}%
\{X_{i}=X_{j}\}$.
\end{remark}

As a consequence, the FPD associated with $(N,M,G)$ is the centralizer of
$\hat{\LL}$ in the space of vector fields defined on open dense diagonal $G$-spaces of $M^{N}$. In
particular it has the structure of a Lie Algebra.
\end{corollary}

\begin{proof}Apply Proposition \ref{equivp1} to the integral formulation \eqref{equivdef}.
\end{proof}

We now aim to describe the components of FPD.
\begin{definition}
Denote by $\SS$ the space of solutions of FPD by components, i.e. the space of functions $f:M^{N-1}\to \chi(M)$ such that 
$$F(X_1,\ldots,X_N)=(F_1(\cdot,X_2,\ldots,X_N),F_2(X_1,\cdot,X_3,\ldots,X_N),\ldots,F_N(X_1,\ldots,X_{N-1},\cdot))$$ is a PD and $f=F_i$ for one of the components, considered as a vector field on $M$ for the $X_i$ variable, with $X_1,\ldots,X_{i-1},X_{i+1},\ldots,X_N$ seen as parameters.
\end{definition}

We have the following main result.
\begin{theorem}\label{t-FPD}
The FPD is a module over the ring of $\hat{\LL}$-jointly radial functions. The space of solutions $\SS$ is also a module over the same ring.

Given $M$ a manifold of dimension $n$ and $f^1,\ldots,f^n$ solutions, independent at each point of an open dense $G$-space, then the module is free and all solutions are of the form $\sum_{i=1}^n \phi^i f^i$ with $\phi^i$ jointly radial functions.
\end{theorem}
\begin{proof} We prove the first statement. Let $F$ be a PD and $\phi$ a jointly radial function. It holds
$$[\phi F, \hat{l}]=\phi [F,\hat{l}]-L_{\hat{l}} \phi F=\phi\cdot 0 - 0 \cdot F=0,$$
i.e. $\phi F$ is a PD. The proof of the second statement is a direct consequence: given $F=(F_1,\ldots,F_N)$ a PD and $f=F_i$ for some $i$ be a solution, then $\phi F$ with $\phi$ jointly radial is a PD and $\phi f=\phi F_i$ is a solution too.

We prove the last statement. Let $F=(F_1,\ldots,F_n)$ be a PD and $f=F_i$ for some $i$ be a solution. Since it is of the form $f:M^{N-1}\to \chi(M)$ and we have a family of $n$ independent $f^i:M^{N-1}\to \chi(M)$, we can always write in a unique way $f=\sum_{i=1}^n \phi^i f^i$ with $\phi^i:M^N\to \R$ being some functions. We need to prove that such $\phi^i$ are jointly radial. We write $[F,\hat l]=0$ and consider its $i$-th component, that reads as
\begin{eqnarray*}
0&=&\sum_{k=1}^N\partial_{X_k}\left(\sum_{j=1}^n\phi^j f^j\right). l(X_k)- \partial_{X_i}l .\left(\sum_{j=1}^n\phi^j f^j\right)=\nonumber\\
&=&\sum_{k=1}^N\sum_{j=1}^n(\partial_{X_k}(\phi^j) .l(X_k)) f^j +\sum_{j=1}^n \phi^j\left(\sum_{k=1}^N \partial_{X_k} f^j. l(X_k)-  \partial_{X_i}l.f^j
\right)\end{eqnarray*}
The last term is zero, since $f^j$ are solutions. By independence of the $f^j$, the first term now becomes $\sum_{k=1}^N(\partial_{X_k}(\phi^j) .l(X_k))=0$, i.e. $\phi^j$ is jointly radial.
\end{proof}

From now on, we will consider solutions of \eqref{maineq} as  smooth vector fields $F\in{\LARGE \chi}(D_{F})$ defined
on a diagonal $G$-space $D_{F}\subset M^N$ that is open and dense. We also identify solutions of \eqref{maineq} if they coincide over an open dense connected diagonal $G$-space.\\


As recalled in the introduction, the case of $N$ agents that are
indistinguishable with respect to the dynamics is very important for modeling
of population dynamics. We call them ``Permutation-Equivariant Population
Dynamics'' from now on. In mathematical terms, this means that permutation of
agents $i,j$ induces the permutation of their trajectories and no variation in
the dynamics of other agents. The formal definition is given here.

\begin{definition} \label{d-PEPD}
We say that a Population Dynamics $(N,M,G,D,F)$ is Permutation-Equivariant
(PEPD for short) if for each $i\neq j\in\{1,\ldots,N\}$ and $k\neq i,j$ it holds
\begin{eqnarray}\label{e-FPEPD}
F_{i}(X_{1},\ldots,X_{i},\ldots,X_{j},\ldots,X_{n})  &  =F_{j}(X_{1},\ldots,X_{j}%
,\ldots,X_{i},\ldots,X_{n}),\label{e-PEPD}\\
F_{k}(X_{1},\ldots,X_{i},\ldots,X_{j},\ldots,X_{n})  &  =F_{k}(X_{1},\ldots,X_{j}%
,\ldots,X_{i},\ldots,X_{n}),\nonumber
\end{eqnarray}
Given $(N,M,G)$ as above, the corresponding Family of Permutation-Equivariant
Population Dynamics (FPEPD for short) is the largest family $\mathbb{F}$ of
smooth vector fields such that for all $F\in\mathbb{F}$ it holds:
\end{definition}

\begin{itemize}
\item the domain of $F$ is a diagonal $G$-space, also being invariant under permutations;

\item $(N,M,G,D,F)$ is a permutation-equivariant population dynamics.
\end{itemize}

Examples of Permutation-Equivariant Population Dynamics are ubiquitous in literature, see e.g. \cite{Un,Trois,Cinq,Sept,Huit,Neuf,CS,Vingt,Vingt5}. We will study examples of PEPDs
in the remaining of the article.\\

 We now define {\it Controlled Population Dynamics}, i.e. dynamics in which an
additional constraint on the dynamics is imposed. This is often the case of
control systems, in which the admissible vector fields are just a subset of
all vector fields on a manifold. As already stated, this is the case of
nonholonomic constraints in the context of mechanical engineering, robotics,
or human locomotion models \cite{Deux,Onze,HelbingLaumond}.

\begin{definition}
\label{d-CPD} We say that $(N,M,G,D,F,\Omega)$ is a \textit{controlled population dynamics} if:

\begin{itemize}
\item $(N,M,G,D,F)$ is a population dynamics;

\item $\Omega$ is a subset of the Lie Algebra of (open-densely defined) vector fields over $M$ and $F\in\Omega$.
\end{itemize}

Given $(N,M,G,\Omega)$ as above, the corresponding \textit{Family of
Controlled Population Dynamics} (controlled FPD for short) is the largest family
$\mathbb{F}$ of smooth vector fields such that for all $F\in\mathbb{F}$ it holds:

\begin{itemize}
\item the domain of $F$ is a diagonal $G$-space;

\item $(N,M,G,D,F)$ is a controlled population dynamics.
\end{itemize}
\end{definition}
\begin{remark}
\label{rem1} It is clear that a controlled FPD is not a Lie algebra in general. However, it is the case for a controlled FPD associated with $(N,M,G,\Omega)$ if the constraint $\Omega$ is a Lie algebra itself.
\end{remark}

\subsection{Families of Population Dynamics with linear group actions} \label{s-linear}

In this section, we study the algebraic condition \eqref{maineq} when $G$ is a linear group acting on $M=\mathbb{R}^{n}$ (or an open dense $G$-invariant subset of it). In this setting, the Lie algebra $\LL$ of $G$
is a Lie subalgebra of $\mathcal{M}_{n}$, the Lie algebra of $n\times n$
matrices. Given $l\in\LL$, we denote by $l.X$ the corresponding linear
vector field over $\mathbb{R}^{n}$. We then write $X=(X_{1},\ldots,X_{N})$ and
$F=(F_{1},\ldots,F_{n})$ with $F_{i}=f_{i}(X)\partial_{X_{i}} $ defined by $N$ functions $f_{i}:D_{F}\rightarrow\mathbb{R}^{n} $. Condition \eqref{maineq} is now restated as the following set of $n$ first order linear partial differential equations on $\mathbb{R}^{nN}$, parametrized by $l\in\LL$: each $f_1,\ldots, f_N$ is a solution $f$ of the following equation: 
\begin{equation}
\partial_{X_{1}}f(X).lX_{1}+\partial_{X_{2}}f(X).lX_{2}+\ldots+\partial
_{X_{N}}f(X).lX_{N}=l.f(X)\qquad\mbox{~for all }l\in\LL. \label{maineq1}
\end{equation}
Following the same requirements for $F$, we only consider smooth solutions $(f_1,\ldots,f_N)$, and we identify two solutions if they coincide over an open dense diagonal $G$-space.

We now aim to describe the set of solutions of \eqref{maineq1}, that is the main objects under study in this article. We then define three important sets.

\begin{definition}
\label{maindef}Given $\mathcal{M}_{n} $ the Lie algebra of $n\times n$ matrices, let $\LL$ be a Lie subalgebra and $\Centr$ its centralizer. 

\begin{itemize}

\item We recall that $\RR$ is the ring of $\hat{\LL}$-jointly radial functions and $\SS$ is the $\RR$-module of solutions, that we identify with the set of functions $\varphi: D_\varphi\to\R^n$ that are solutions of \eqref{maineq1}. 

\item We denote by $\CC$ the set of solutions given by the centralizer $\Centr$, i.e. the $\RR$-module generated by functions
$$\left\{c .X_i\mbox{~~such that~~}i\in\{1,\ldots,N\}\mbox{ and }c \in\Centr\right\},$$
where we denote by $c.X_i$ the function $(X_1,\ldots,X_N)\mapsto c .X_i$.

\item We denote by $\EE$ the set of elementary solutions, i.e. the $\RR$-module generated by 
$\left\{X_1,\ldots, X_N\right\},$
where we denote by $X_i$ the function $(X_1,\ldots,X_N)\mapsto X_i$.

\end{itemize}
\end{definition}

We now state some useful properties for these sets.

\begin{proposition}
\label{mainprop} Let $\LL,\Centr,\RR,\SS,\CC,\EE$ be defined as in Definition \ref{maindef}. Then, it holds  $\EE\subset\CC\subset\SS$, i.e. elementary solutions and solutions given by the centralizer are actually solutions.
\end{proposition}
\begin{proof} The inclusion $\EE\subset \CC$ is a direct consequence of the fact that $\Id\in\Centr$. For $\CC\subset \SS$, fix $c\in \Centr$ and define $f=c.X_{i}$. Then, we rewrite \eqref{maineq1} for $f$ as $c.l.X_{i} -l.(c.X_{i})=0$ for all $l\in \LL$. This is exactly the condition for $c$ being in the centralizer of $\LL$ in $\mathcal{M}_{n}$, see Definition \ref{d-centralizer}. As a consequence, the generators of $\CC$ belong to $\SS$, thus the inclusion is a consequence of the fact that $\SS$ is a $\RR$-module.
\end{proof}

We recall that our goal is to describe all FPD associated to a given triple $(N,M,G)$, that is equivalent to describe the set $\SS$ of all solutions of \eqref{maineq1}. It is then very useful to study the case in which all solutions are reduced either to solutions given by the centralizer (i.e. $\CC=\SS$) or, even better, to elementary solutions (i.e. $\EE=\SS$). In both cases, the advantage is that the description of $\CC$ or $\EE$ is somehow explicit. The main result of this section, that we now state, describes some important cases in which one of the two identities holds.

\begin{theorem} \label{mainth} Let $G$ be a Lie group acting linearly on $M=\R^n$. Let $\LL,\Centr,\RR,\SS,\CC,\EE$ defined as in Definition \ref{maindef}. Then:

\begin{enumerate}[{\bf {I}tem 1.}]
\item If the space $\{c.X,c\in\Centr\}$ has rank $n$ at some point
$X\in\mathbb{R}^{n} $, then $\SS=\CC$ and it is a free $\RR$-module; moreover, given $\{c_1,\ldots,c_n\}$ independent generators of $\Centr$ and $i\in\{1,\ldots,N\}$ a fixed index, a basis of $\SS$ is given by $\{c_1X_i,\ldots,c_nX_i\}$.

\item If $N\geq n$, then  $\SS=\EE$ and it is a free $\RR$-module, with basis $\{X_{i_1},\ldots,X_{i_n}\}$ with $i_1,\ldots,i_n$ being any set of $n$ distinct indexes in $\{1,\ldots,N\}$.
\end{enumerate}
\end{theorem}
\begin{remark} The choice of the basis in Item 2 has a direct impact on the ring $\RR$: its elements are functions whose domain is a $G$-space, that depends on the selection of the basis. In any case, the FPD is completely characterized on such domain.
\end{remark}
\begin{proof}
{\bf Item 1}. Let $c_{1},\ldots,c_{n}$ be independent generators of $\Centr$ such that
$\{c_{i}.X\}$ has rank $n$ at some $X$, then on an open set $\mathcal{O}\subset \R^n$, which is also
dense due to the algebraic character of linear dependence. Since $c_1\cdot X_i,\ldots,c_n \cdot X_i$ are all solutions due to Proposition \ref{mainprop} and are independent by hypothesis, we apply Theorem \ref{t-FPD} and get the result.

{\bf Item 2.} For simplicity of notation, we choose $\{i_1,\ldots,i_n\}$ to be the first $n$ indexes. Consider the set $\mathcal{O}\subset \R^{nN}$ of $X=(X_1,\ldots, X_n,X_{n+1},\ldots,X_N)$ such that $X_1,\ldots,X_n\in \R^n$ are $n$ independent vectors. It is clear that the set $\mathcal{O}$ is open dense. Restrict $F$ to $\mathcal{O}$ and, similarly to Item 1, observe that $X_1,\ldots,X_n$ are independent and are solutions. Apply again Theorem \ref{t-FPD} and get the result.
\end{proof}

\begin{remark}
We see that, for small $N $, it may happen that $\SS$ properly contains
$\mathcal{E}$. It is the case for instance if $N=1$ and $\Centr$ is not
 reduced to $\R.\Id $, which holds for instance when $G$ is the group
$SO(2,\mathbb{R})$ acting on the space $\R^2\setminus\{0\}$. See Remark \ref{r-proper-inclusion-1} below.
\end{remark}

\subsubsection{Time reparametrization}

We now discuss the problem of time reparametrization of PD. Indeed, it is clear that, if $F$ is a PD, then $\psi F$ is not a PD, in general. However, Proposition \ref{mainprop} implies that $\psi F$ is a PD as soon as $\psi$ is $\hat{\LL}$-jointly radial. This is particularly interesting for Euclidean and pseudo-Euclidean manifolds, that admit a natural arclength parametrization of trajectories.
\begin{corollary} \label{c-reparam}
Let $G$ be a linear group acting on $M=\R^n$. Let $F$ be a Population Dynamics, and $\psi$ be a $\hat{\LL}$-jointly radial function. Then $\psi F$ is a Population Dynamics.

Moreover, denote by $\|F\|$ the standard product Euclidean norm on $\R^{nN}$ and let $G$ be a group of linear isometries. Then, the functions $\phi=\|F\|^\alpha$ are jointly radial for any $\alpha\in\mathbb{Z}$. In particular, the vector field $\tilde F:=\frac{F}{\|F\|}$ is again a PD, outside equilibria of $F$. The trajectories $\tilde X(s)$ of $\tilde F$ (that geometrically coincide with trajectories $X(t)$ of $F$ outside equlibria) are parametrized by arclength $ds=\sqrt{ \sum_{i=1}^{N} ||F_{i}(X(t))||^{2}}dt$.

The same results hold when $M$ is a pseudo-Euclidean space and $G$ is a group of linear pseudo-isometries. Reparametrization of trajectories only holds if $\|F\|^2>0$. 
\end{corollary}
\begin{proof} The first statement is a consequence of the fact that $\SS$ is a $\RR$-module, proved in Theorem \ref{t-FPD}.

We prove the second statement for $\alpha=2$ only, the proof for other values being a direct consequence. Denote by $<.,.>$ the standard Euclidean scalar product on $\R^{nN}$. Since $G$ is a linear isometry, then its Lie algebra is given by elements of the form $AX$ with $A$ skew-symmetric. It then holds
$$L_{\widehat{AX}} (\|F\|^2)=L_{\widehat{AX}} (<F,F>)=2<L_{\widehat{AX}} F,F>=2\sum_{i=1}^N<A\cdot F_i,F_i>=0.$$\end{proof}

\subsubsection{Permutation-Equivariant Population Dynamics with linear group action} We now discuss the structure of Permutation-Equivariant Population Dynamics, when $G$ is a linear group acting on $M=\R^n$. By using the notation above, we write $X=(X_{1},\ldots,X_{N})$ and
$F=(F_{1},\ldots,F_{n})$ with $F_{i}=f_{i}(X)\partial_{ X_{i}} $ defined by $N$ functions $f_{i}:D_{F}\rightarrow\mathbb{R}^{n} $. It is clear that the vector field $F$ must be a PD, hence that \eqref{maineq} needs to be satisfied. Moreover, it is also clear that, as soon as one among the components $f_i$ is defined, all other components $f_j$ with $j\neq i$ are directly defined by \eqref{e-PEPD}. For a similar reason, the structure of $f_i$ needs to be permutation-equivariant with respect to indexes different from $i$, i.e.
\begin{equation}
f_i(X_1,\ldots,X_j,\ldots,X_k,\ldots,X_N)=f_i(X_1,\ldots,X_k,\ldots,X_j,\ldots,X_N)\hspace{3mm}\mbox{for all}\hspace{1mm}j\neq k\hspace{1mm}\mbox{with}\hspace{1mm}j,k\neq i. 
\label{e-PEdef}
\end{equation}

A direct consequence of this discussion is the following result.
\begin{corollary} Let $G$ be a linear Lie group on $M=\R^n$. Let $\LL,\Centr$ be defined as in Definition \ref{maindef}. Let the space $\{c.X,c\in\Centr\}$ have rank $n$ at some point, with generators $c_1X,\ldots,c_n X$.

Denote by $\check{X}_i:=(X_1,\ldots,X_{i-1},X_{i+1},\ldots,X_N)$, i.e. the configuration of all agents except $X_i$. Then, the Family of PEPD is given by $F=(F_{1},\ldots,F_{n})$ with $F_{i}=f_{i}(X)\partial_{X_{i}}$, where 
$$f_i(X_i,\check{X}_i)=c_1X_i\psi_1(X_i,\check{X}_i)+\ldots +c_nX_i\psi_n(X_i,\check{X}_i)$$
and the $\psi_k$ are analytic functions of $N$ variables, jointly radial on all variables and permutation invariant in the last $N-1$ ones.
\end{corollary}

\subsection{Translation equivariance in Euclidean and pseudo-Euclidean spaces} \label{s-translation}

In this section, we study two key aspects of the dynamics on Euclidean and pseudo-Euclidean spaces, in which the group $G$ of equivariance contains translations. The case is very relevant for models of population dynamics, as already described in the introduction. In this setting, we extend the previous results, moreover reducing the number of variables from $N$ to $N-1$.

We then assume to have a group $G=K\ltimes\mathbb{R}^{n}$ being the semi-direct product of a connected Lie group $K$ by the Euclidean space $M=\mathbb{R}^{n}$. We assume the following action: given $g=(M,b)$ and $x\in\R^n $, it holds $g.x:=Mx+b$. As a consequence, equivariance with respect to $G$ includes equivariance with respect to translations. This in turn forces all $F_{i}$ to be of the following form:
\begin{equation}
F_{i}=F_{i}(X_{1}-X_{j},\ldots,X_{N}-X_{j})\text{ for some }j. \label{semidvec}%
\end{equation}
This already has a simple consequence.
\begin{proposition} \label{p-1}
For $N=1$ (a single particle), the FPD associated to $(1,\mathbb{R}^{n},K\ltimes\mathbb{R}^{n})$ contains constant vector fields only (eventually being smaller).
\end{proposition}
\begin{proof} A Population Dynamics $F$ necessarily has to be independent on $X_1$, due to \eqref{semidvec}, i.e. constant.\end{proof}

From now on in this section, for simplicity in exposition, we chose $j=1$
in \eqref{semidvec}. We then define the difference variables with
respect to $X_{1}$, that are coordinates in $\R^{n(N-1)}$, as follows:%
\begin{equation}\label{e-Z}
Z=(Z_{2},\ldots,Z_{N}):=(X_{2}-X_{1},\ldots,X_{N}-X_{1}).
\end{equation}

We are now ready to restate condition \eqref{maineq} for vector fields $F$ being jointly equivariant with respect to the action of $K$ only, since equivariance with respect to translations is taken into account by \eqref{semidvec}. We again write $F=(F_{1},\ldots,F_{n})$ with $F_{i}=f_{i}(X)\partial_{X_{i}} $ defined by $N$ functions $f_{i}:D_{F}\rightarrow\mathbb{R}^{n} $. Since $K$ acts linearly on $\R^N$, then each $f_i$ solves \eqref{maineq1} for all $l\in\KK$, where $\KK$ is the Lie algebra of $K$. We now recall that $f=f(Z)$, hence it is easy to observe that $\partial_{X_1}f=-\sum_{i=2}^n \partial_{Z_i} f$, while $\partial_{X_j} f=\partial_{Z_j} f$ for $j\neq 1$. Then \eqref{maineq1} in this context reads as
\begin{equation*}
\left(-\sum_{i=2}^n\partial_{Z_{i}}f(Z).lX_{1}\right)+\partial_{Z_{2}}f(Z).lX_{2}+\ldots+\partial
_{Z_{N}}f(Z).lX_{N}=l.f(Z)\qquad\mbox{~for all }l\in\KK,
\end{equation*}
hence by linearity
\begin{equation}
\partial_{Z_{2}}f(Z).lZ_{2}+\ldots+\partial
_{Z_{N}}f(Z).lZ_{N}=l.f(Z)\qquad\mbox{~for all }l\in\KK.\label{e-PDEz}
\end{equation}
This equation has the same structure of \eqref{maineq1} with respect to the variables $Z_2,\ldots,Z_N$ and with Lie algebra $\KK$. Then, results of Section \ref{s-linear} can be translated to this context.

\begin{corollary} Time reparametrization in Euclidean and pseudo-Euclidean spaces described in Corollary \ref{c-reparam} holds if $G$ is an affine (pseudo-)isometry too.
\end{corollary}

\begin{corollary} \label{c-Z} Let $G=K\ltimes\mathbb{R}^{n}$ be the semi-direct product of a connected Lie group $K$ by the Euclidean space $M=\mathbb{R}^{n}$, acting as follows: given $g=(M,b)$ and $x\in\R^n $, it holds $g.x:=Mx+b$.  Let $\KK$ be the Lie algebra of $K$ and $\Centr$ its centralizer. Let $\RR$ be the ring of $\hat{\KK}$-jointly radial functions of $N-1$ variables $Z_2,\ldots,Z_N$, that is the Abelian ring of analytic functions $\varphi:D_\varphi\subset
\mathbb{R}^{n(N-1)}\rightarrow\mathbb{R}$, with open dense diagonal $K$-space $D_\varphi$ that are radial with respect to the diagonal action of $K$. Equivalently, by Proposition \ref{proprad}, jointly radial functions are solutions of
\begin{equation}
L_{\widehat{lZ}}\varphi=0\mbox{~~for all~~} \hat{l}\in\hat\KK.\label{e-jointlyradialZ}
\end{equation}

Denote by $\SS$ the set of functions $\varphi: D_\varphi\to\R^n$ that are solutions of \eqref{e-PDEz}. Denote by $\CC$ the set of solutions given by the centralizer $\Centr$, i.e. the $\RR$-module generated by functions
$$\left\{c .Z_i\mbox{~~such that~~}i\in\{2,\ldots,N\}\mbox{ and }c \in\Centr\right\}.$$ Denote by $\EE$ the set of elementary solutions, i.e. the $\RR$-module generated by $\left\{Z_2,\ldots, Z_N\right\}.$

Then: \begin{enumerate}[{\bf {I}tem 1.}]
\item If the space $\{c.Z,c\in\Centr\}$ has rank $n$ at a point
$Z\in\mathbb{R}^{n} $, then $\SS=\CC$ and it is a free $\RR$-module; moreover, given $\{c_1,\ldots,c_n\}$ independent generators of $\Centr$ and $i\in\{2,\ldots,N\}$ a fixed index, a basis of $\SS$ is given by $\{c_1Z_i,\ldots,c_nZ_i\}$.

\item If $N-1\geq n$, then  $\SS=\EE$ and it is a free $\RR$-module, with basis $\{Z_{i_1},\ldots,Z_{i_{n}}\}$ with $i_1,\ldots,i_{n}$ being any set of $n$ distinct indexes in $\{2,\ldots,N\}$.
\end{enumerate}
\end{corollary}

In both cases, proofs are direct translations of the corresponding results.

We also consider the case of Permutation-Equivariant Population Dynamics. We have the following result.

\begin{corollary}[Translation-equivariant PEPDs] \label{c-trasl-PEPD} Let $G=K\ltimes \R^n$ be as in Corollary \ref{c-Z}, with $\KK$ the Lie algebra of $K$ and $\Centr$ its centralizer. Let the space $\{cX,c\in\Centr\}$ have rank $n$ at some point
$X\in\mathbb{R}^{n}$,  with $c_1X,\ldots,c_nX$ being a basis. Define
$$\check{Z}_i:=(X_1-X_i,\ldots,X_{i-1}-X_i,X_{i+1}-X_i,\ldots,X_{N} - X_i).$$

Then, the FPEPD is given by $F=\sum_{i=1}^N H(\check{Z_i})\partial_{X_i}$, where  
$$H(V_2,\ldots,V_N)=c_1(V_2+\ldots+V_N)\phi_1(V_2,\ldots,V_N)+\ldots+ c_n(V_2+\ldots+V_N)\phi_n(V_2,\ldots,V_N)$$
with $\phi_1,\ldots,\phi_n$ being analytic, jointly radial and permutation invariant. The formula holds in the $G$-diagonal invariant set $\R^{nN}\setminus(\cup_{i=1}^N\{(X_1+\ldots+X_N)-NX_i= 0\})$.
\end{corollary}
\begin{proof} If $\{cX,c\in\Centr\}$ has rank $n$ in a point, then it has rank $n$ in an open dense set, as in the proof of Theorem \ref{mainth}. By removing the union of hyperplanes $\cup_{i=1}^N\{(X_1+\ldots+X_N)-NX_i= 0\}$, we are left again with an open dense set $\mathcal{O}\subset \R^n$. Outside this set, for each $i=1,\ldots,N$ it holds that 
$$c_1(V_2+\ldots+V_N),\ldots,c_n(V_2+\ldots+V_N)$$
is a basis of $\R^n$ for any choice of $\check{Z}_i=(V_2,\ldots,V_N)$. By writing $F=\sum_{i=1}^N F_i(\check{Z_i})\partial_{X_i}$, it holds $$F_1(\check{Z}_1)=c_1(Z_{2,1}+\ldots+Z_{N,1})\phi_1(\check{Z}_1)+\ldots+c_n(Z_{2,1}+\ldots+Z_{N,1})\phi_n(\check{Z}_1),$$
for some functions $\phi_1,\ldots,\phi_N$, as in the proof of Theorem \ref{mainth}, Item 1, where $Z_{j,k}:=X_j-X_k$. Since the term $Z_{2,1}+\ldots+Z_{N,1}$ is permutation invariant, then $F_1$ is permutation invariant when the $\phi_i$ are permutation invariant too. This completely determines $H(\check{Z}_1):=F_1(\check{Z}_1)$. Then, $F_2,\ldots,F_N$ are completely determined by permutation equivariance, i.e. $F_i(\check{Z}_i)=H(\check{Z}_i)$.
\end{proof}
%
%
%
%
%

\subsection{Population Dynamics on Lie groups} \label{s-GM}

In this section, we describe Population Dynamics in the following interesting case: the Lie group $G$ is both the state space for each agent and the group of equivariance. In this case, the manifold is then $M=G^N$, where $N$ is the number of agents. We will provide an example in Section \ref{s-S1}, in which the group is $G=SO(2,\R)$ the group of rotations of the plane, that is identified with the circle $S^1$. The result will also play a role in studying the Constrained Population Dynamics of Section \ref{s-unicycle}.

We have the following description of Families of Population Dynamics.

\begin{proposition}[FPDs on Lie groups] \label{p-LieGroups} Let $G=M$ be a Lie group, acting on itself by left translation: $G\times M\to M$ maps $(g,h)$ to $gh$. Then, a function $\phi:G^N\to\R$ is jointly radial if and only if it satisfies 
$$\phi(g_{1},\ldots,g_{N})=\psi(g_{1}^{-1}g_{2},\ldots,g_{1}^{-1}g_{N})$$
for an analytic function $\psi:G^{N-1}\to \R$.

Consider the FPD associated to $(N,M,G)$. It is composed by vector fields $F=\sum_{i=1}^{N} F_i$, where 
$$F_{i}(g_{1},\ldots,g_{N})=\sum_{j=1}^{n} v_{j}(g_{i})\phi_{ij}(g_{1},\ldots,g_{N}),$$
where $v_1,\ldots,v_n$ is a basis of the Lie algebra of left-invariant vector fields of $G$ and functions $\phi_{ij}$ are jointly-radial.

\end{proposition}
\begin{proof} The first statement is a direct consequence of Proposition \ref{proprad}, written with respect to the Lie Algebra $\hat \LL$ of the diagonal extension of left-invariant vector fields.

The second statement is a particular case of Theorem \ref{t-FPD}, after proving that each $F_i(g_1,\ldots,g_N):=v_j(g_i)$ for $j=1,\ldots,n$ is a solution. Since $G$ acts by left translations, its Lie algebra is given by right-invariant vector fields. Let $l$ be one of such right-invariant vector fields and define its diagonal extension $\hat l(g_1,\ldots,g_N)=\sum_k l(g_k)$, where $l(g_k)$ acts on the $k$-th agent. By studying the $i$-th component of $[F,\hat l]$, it holds  $[v_j(g_i),l(g_i)]=0$, since each left-invariant vector field $v_j$ commutes with right-invariant vector fields, see Proposition \ref{p-LR}. Thus, $F_i$ is a solution.\end{proof}

\subsection{Gradient flow Population Dynamics}

In this section, we describe Population Dynamics that are moreover gradient flows. This means that the dynamics $\dot X=F(X)$ is indeed of the form $\dot X=\nabla \phi(X)$ for some potential function $\phi:M^N\to \R$. It is clear that the standard setting for such dynamics is given by Riemannian manifolds, where the natural definition of the gradient is available.

The main result of this section is the following.
\begin{proposition} \label{p-GF} Let $M$ be a Riemannian manifold and $G$ a connected Lie group of isometries of $M$.

Let $\phi:M^N\to \R$ be a jointly-radial function. Then, its gradient is a jointly-equivariant vector field.
\end{proposition}
\begin{proof} It is easy to observe that the result for $G$ acting diagonally on $M^N$ is equivalent to the case of $G':=G^N$ acting on $M':=M^N$. Then, from now on we only consider the case $N=1$, hence $\hat l=l$.

We need to prove that $[\nabla\phi, l]=0$ for all $l\in \LL$. We fix a point $X\in M$ and define normal coordinates around it: each point in a small neighborhood $\mathcal{N}_X$ of $X$ is written in coordinates $x=(x_1,\ldots, x_n)\in \R^n$ so that the Riemannian tensor $g$ at $x$ is $g_{ij}(x)=\delta_{ij}-\frac13 \sum_{kl}R_{ijkl}x_kx_l+o(\|x\|^2)$. In particular, we have that Christoffel symbols are zero. The inverse Riemannian tensor has a similar structure: $g^{ij}(x)=\delta^{ij}+o(\|x\|)$.

Write $\nabla f=\sum_i\partial_{i} f(x) g^{ij}(x)\partial_{i}$ and $l=\sum_i l_i(x)\partial_{i}$. Here, the symbol $\partial_i$ is the derivative with respect to the variable $x_i$ in normal coordinates. Since $f$ is radial, then $L_{l}f=0$, i.e. $g_x(\nabla f(x), l(x))=0$ for all $x\in M$ due to the definition of gradient. By writing the identity in normal coordinates, we have $\psi(x)=\sum_i (\partial_i f(x))l_i(x)+o(\|x\|)=0$. Since $\psi(x)=0$ for all $x\in M$, then its differential is zero. By computing it in coordinates, we have $\sum_{i} (\partial^2_{ji} f(x))l_i(x)+(\partial_{i} f(x))(\partial_{j}l_i(x))+o(1)=0$. By sending $x\to 0$, we have $o(1)=0$, thus for all indexes $j$ the following identity in $X$ holds:
\begin{equation}
[(\partial^2_{i j} f)  l_i + (\partial_i f) (\partial_{j} l_i)]_{|_{x=0}}=0.\label{e-der2}
\end{equation}

We now prove that the matrix $\partial_j l_i$ is skew-symmetric at $X$. We write $l(x)=L_0+L_1 x+o(\|x\|)$ in Taylor series, thus $L_1=\partial_j l_i$. Its exponential is then $\exp(t l(x))=L_0 t+ (\Id+L_1 t) x+o(x,t)$, thus by differentiation in $x$ we have $D\exp(t l(x))= (\Id+L_1 t)+o(1,t)$. The fact that $G$ is a group of isometry reads as 
$$g_x(v,w)=g_x(D\exp(t l(x)) v, D \exp(t l(x)) w).$$
By differentiating the identity with respect to $t$ at $t=0$, we have 
$$0=g_x(L_1 v,w)+g(x)( v,L_1w).$$
By letting $x\to 0$, the Riemannian tensor becomes the Euclidean one, thus the identity becomes
$$w^T L_1 v + v^T L_1 w=0,$$
i.e. $L_1=\partial_j l_i$ is skew-symmetric.

We now use the fact that $\partial_{j} l_i$ is skew-symmetric in \eqref{e-der2}. It then holds
$$[(\partial^2_{i j} f)  l_i - (\partial_{i} l_j) (\partial_i f)]_{|_{x=0}} =0,$$
for all indexes $j$. This is the formula for $[\nabla f, l]$ in coordinates at $X$. It then holds $[\nabla f, l]=0$ at $X$. Since the choice of $X$ was arbitrary, the identity holds for all $X\in M$.
\end{proof}
\begin{corollary} \label{c-perp} Let $M$ be a Riemannian manifold of dimension $n=2$ and $N\geq 1$. Let $F=(F_1,\ldots,F_N)=\nabla \phi$ be the gradient of a jointly radial function. Define $F^\perp:=((F_1)^\perp,\ldots,(F_N)^\perp)$, where $(F_i)^\perp$ is a continuous choice of perpendicular vectors to $F_i$ at each point of $M$. Then, $F^\perp$ is well defined and is a PD.
\end{corollary}
\begin{proof} Assume that $F$ has a connected (open dense) domain. As in the previous proof, we use normal coordinates at each point of $(X_1,\ldots,X_N)\in M^N$, that coincide with the normal coordinates on each component. Write $X_i=(x_i,y_i)$ in normal coordinates, then $F(X)=\sum_{i=1}^N f_{i}\partial_{x_{i}}+g_{i}\partial_{y_{i}}$. We then define $F(X)^\perp=\sum_{i=1}^N -g_i\partial_{x_i}+f_i\partial_{y_i}+o(1)$, where we used the fact that the Riemannian metric in normal coordinates coincides with the Euclidean one up to order 2. In other terms, for each $F_i$ we choose $(F_i)^\perp=JF_i$ with $J$ being the symplectic matrix
\begin{equation}\label{e-J}
J=\left(\begin{array}{cc} 0 & -1 \\ 1 & 0 \end{array}\right).
\end{equation}
We then write the condition $[F^\perp,\hat l]=0$ for each coordinate, that reads as
\begin{equation}\label{e-perp}
(\partial_{X_1} JF_i) l(X_1)+\ldots+(\partial_{X_N} JF_i) l(X_N)-(\partial_{X_i}l) JF_i=0.
\end{equation}
By recalling that $\partial_{X_i}l=L_1$ with $L_1$ skew-symmetric (as proved in Proposition \ref{p-GF}), we have that it commutes with $J$, due to the fact that we deal with 2-dimensional skew-symmetric matrices; $L_1$ is indeed a multiple of $J$. We can then rewrite \eqref{e-perp} as $J \left([F,\hat l]\right)_i=0$, that is true due to equivariance of $F$.

More generally, for each connected component of the domain of $F$ and each coordinate $F_i$, we can choose $(F_i)^\perp$ as either $JF_i$ or $-JF_i$. The proof is identical, since the condition \eqref{e-perp} holds by replacing $J$ with $-J$.
\end{proof}

We will use these results throughout the remaining of the article, to identify some Population Dynamics that are gradient flows. It is remarkable to observe that it might happen that other Population Dynamics are gradient flows, even though they are not gradient of jointly radial functions. We also have the following useful corollary.

\begin{corollary} \label{c-distance} Let $M$ be a Riemannian manifold of dimension $n$ and $G$ a connected Lie group of isometries of $M$. If $N\geq n+1$, a basis of the space of solutions $\SS$ is given by the component of the vector fields $\nabla d(x_1,x_2),\ldots, \nabla d(x_1,x_{n+1})$ for the variable $x_1$, where $d$ is the Riemannian distance.
\end{corollary}
\begin{proof} It is clear that the vector fields are independent. Since they are gradient flows, they are jointly equivariant. Then, their first components are both independent functions and solutions. Since we have $n$ independent functions for the $n$-dimensional manifold, it is a basis due to Theorem \ref{t-FPD}.
\end{proof}

\section{Population Dynamics on the Euclidean plane}
\label{s-SE2}

In this section, we describe Population Dynamics on the Euclidean plane, being jointly equivariant with respect to rototraslations of the plane. We already stated in the introduction that this setting is often used for several important applications. We first describe the general FPD, by showing the role of each functional parameter. We finally describe gradient flows and Permutation-Equivariant PDs.

The setting of this section is then the following: the manifold is $M=\R^2$ and the group is the special Euclidean group of rototranslations $G=SO(2,\mathbb{R})\ltimes \R^2 $. We use standard coordinates on $\R^{2N}$, that we denote by  $X_{i}=(x_{i},y_{i})$ for $i=1,\ldots,N$, and denote the standard Euclidean norm by $\|\cdot\|$.

The case $N=1$ is easy but already interesting.
\begin{proposition} The FPD associated to $(N=1,\mathbb{R}^{2},SO(2,\R)\ltimes
\mathbb{R}^{2})$ contains the null vector field only.
\end{proposition}
\begin{proof} By Proposition \ref{p-1}, PDs for $N=1$ are contained in the set of constant vector fields $f(x)=v\in\R^2$. By adding equivariance with respect to $SO(2)$, we impose $[Jx,v]=0$. This implies $v=0$.
\end{proof}

We then consider $N>1$ from now on. We introduce difference variables with respect to $X_1$, that we also write in polar coordinates, i.e. 
\begin{equation}\label{e-Zpolar}
Z_i:=X_i-X_1= \rho_i(\cos(\theta_i),\sin(\theta_i)),
\end{equation}
where $\rho_i:=\|X_i-X_1\|$ and $\theta_i\in[0,2\pi)$ is the argument of $Z_i$.

We have the following result.

\begin{proposition}\label{p-SE2} Consider the FPD associated to $(N,\mathbb{R}^{2},SO(2,\R)\ltimes
\mathbb{R}^{2})$ with $N>1$. Assume $(x_1-x_2,y_1-y_2)\neq (0,0)$ for simplicity of notation. Then, the FPD is composed of all vector fields of the form 
$$F=F_{1}%
+\ldots+F_{N}\qquad \mbox{~~with~~}\qquad F_{i}=f_{i,1}\partial_{x_{i}}+f_{i,2}%
\partial_{y_{i}},$$ and 
\begin{equation}\label{e-PD-SE2}
\left(
\begin{array}
[c]{c}%
f_{i,1}\\
f_{i,2}%
\end{array}
\right)  =\left(%
\begin{array}
[c]{cc}%
x_{2}-x_{1} & y_{1}-y_{2}\\
y_{2}-y_{1} & x_{2}-x_{1}%
\end{array}
\right)\binom{\lambda_{i}(\rho_{2},\ldots,\rho_{n},\theta_{3}-\theta_{2},\ldots,\theta
_{N}-\theta_{2})}{\mu_{i}(\rho_{2},\ldots,\rho_{n},\theta_{3}-\theta
_{2},\ldots,\theta_{N}-\theta_{2})},
\end{equation}
where functions $\lambda_{i},\mu_{i},$ are arbitrary analytic functions with a domain being a diagonal $SO(2,\R)$-space.
\end{proposition}
\begin{proof}
The centralizer $\Centr$ of $SO(2,\R)$ is $\Centr=\R.\Id+\R. J$, where $J$ is the symplectic matrix \eqref{e-J}.

Since $\Centr$ has dimension 2, one can apply Corollary \ref{c-Z}, Item 1. This gives the result.
\end{proof}
\begin{remark}
\label{remmain} The number $2N$ of functional parameters here is minimal, as each $F_i\in\SS$ is independent and $\SS=\CC$ has dimension $2$. Moreover, it depends on $2N-3$ variables, that is the minimal number of variables too.
\end{remark}

In the formula above, it is clear that both $X_1$ and $X_2$ play specific (but different) roles. We then now give examples of dynamics for a system of $N=2$ particles only. In this case, we can describe the physical meaning of the functional parameters: $\lambda_i$ plays the role of the radial component of the interaction, while $\mu_i$ is the rotational component. In particular:
\begin{itemize}
\item $\lambda_1>0$ promotes convergence, and similarly for $\lambda_2<0$. Opposite signs promote distancing.
\item $\mu_1>0$ promotes clockwise rotation, and similarly for $\mu_2<0$. Opposite signs promote counter-clockwise rotation.
\end{itemize}
This is made clear from the following examples. Trajectories of these examples can be found in Figure \ref{f-A}. Agent 1 is depicted in green, Agent 2 is depicted in blue. The initial state is always $X_1=(0,0)$, $X_2=(1,1)$.

\begin{description}
\item[Example A.1] Set $\mu_1=\mu_2=0$. For $\lambda_1=1$ and $\lambda_2=-1$ we have convergence to a common state.
\item[Example A.2] Again with $\mu_1=\mu_2=0$, for $\lambda_1=1$ and $\lambda_2=0.5$, agent 1 aims to converge towards agent 2, that in turn aims to distance itself (with a smaller velocity).
\item[Example A.3] Again with $\mu_1=\mu_2=0$, for $\lambda_1=-1+\rho_2$ and $\lambda_2=1-\rho_2$, agents aim to converge towards a configuration with distance 1.
\item[Example A.4] Again with $\mu_1=\mu_2=0$, for $\lambda_1=-1+\rho_2$ and $\lambda_2=0.5-\rho_2$, agents aim to converge towards a configuration with different distances (1 and 0,5 respectively). Thus, agent 1 starts converging, then distances itself after reaching distance 1.
\item[Example A.5] Set $\lambda_1=\lambda_2=0$. For $\mu_1=1$ and $\mu_2=0$ we have clockwise rotation of agent 1 around agent 2.
\item[Example A.6] Set $\lambda_1=\lambda_2=0$. For $\mu_1=1$ and $\mu_2=1$ we have parallel displacement, as agent 1 moves clockwise and agent 2 moves counter-clockwise.
\item[Example A.7] Set $\lambda_1=\lambda_2=0$. For $\mu_1=1$ and $\mu_2=-1$ we have clockwise movement.
\item[Example A.8] For $\lambda_1=1$, $\lambda_2=-1$, $\mu_1=1$ and $\mu_2=-1$ we have clockwise movement, together with convergence.
\item[Example A.9] For $\lambda_1=1$, $\lambda_2=0$, $\mu_1=0$ and $\mu_2=1$ we have clockwise movement of agent 2 with respect to agent 1, together with convergence of agent 1 towards agent 2 via a radial movement.
\item[Example A.10] For $\lambda_1=-1+\rho_2$ and $\lambda_2=1-\rho_2$, with $\mu_1=1,\mu_2=-1$ agents aim to converge towards a configuration with distance 1, together with clockwise rotation.
\item[Example A.11] For $\lambda_1=-1+\rho_2$ and $\lambda_2=0.5-\rho_2$, with $\mu_1=1,\mu_2=-1$ agents aim to converge towards a configuration with different distances, together with clockwise rotation. The result is a convergence to a limit clockwise trajectory, with distance being the average of the desired distances.
\item[Example A.12] For  $\lambda_1=-0.5+\rho_2$ and $\lambda_2=0.5-\rho_2$, with $\mu_1=1-\rho_2,\mu_2=-1+\rho_2$ agents aim to converge towards a configuration with distance 0.5. This first induces counter-clockwise rotation, that then turns into clockwise rotation.
\end{description}

\begin{figure}[htb]
\centering
\begin{subfigure}{0.23\textwidth}
\includegraphics[width=\textwidth]{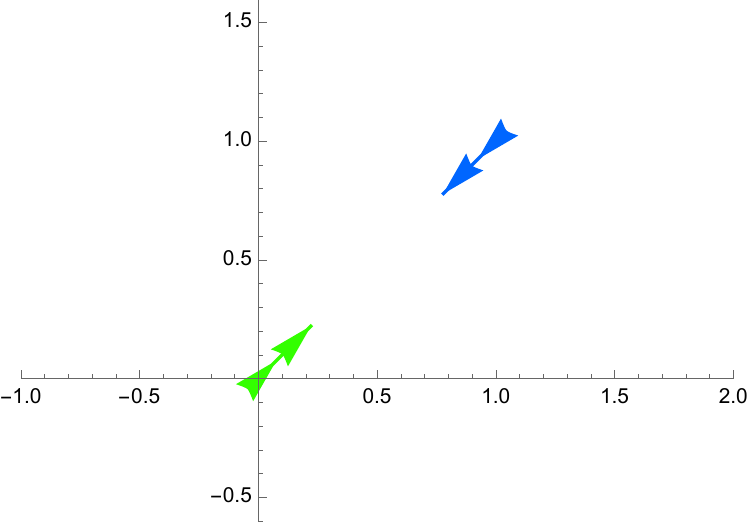}%
\caption{Example A.1: $\lambda_1=1,\lambda_2=-1$,\\ $\mu_1=\mu_2=0$.}
\end{subfigure}
\hfill
\begin{subfigure}{0.23\textwidth}
\includegraphics[width=\textwidth]{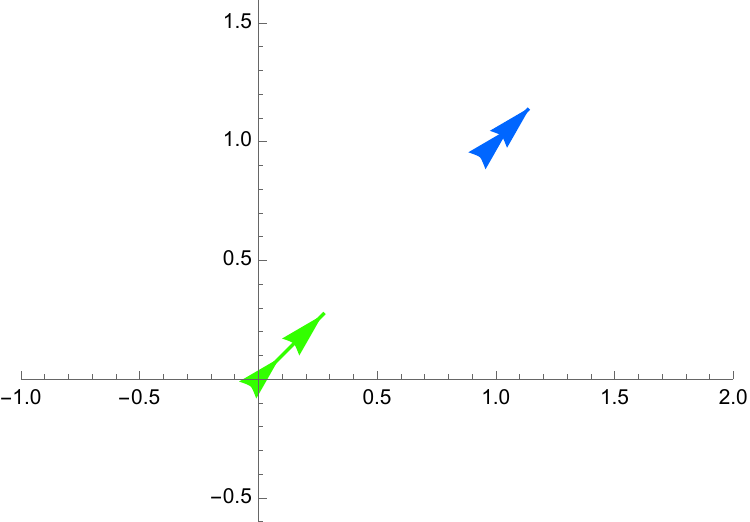}%
\caption{Example A.2: $\lambda_1=1,\lambda_2=0.5$,\\ $\mu_1=\mu_2=0$.}
\end{subfigure}
\hfill
\begin{subfigure}{0.23\textwidth}
\includegraphics[width=\textwidth]{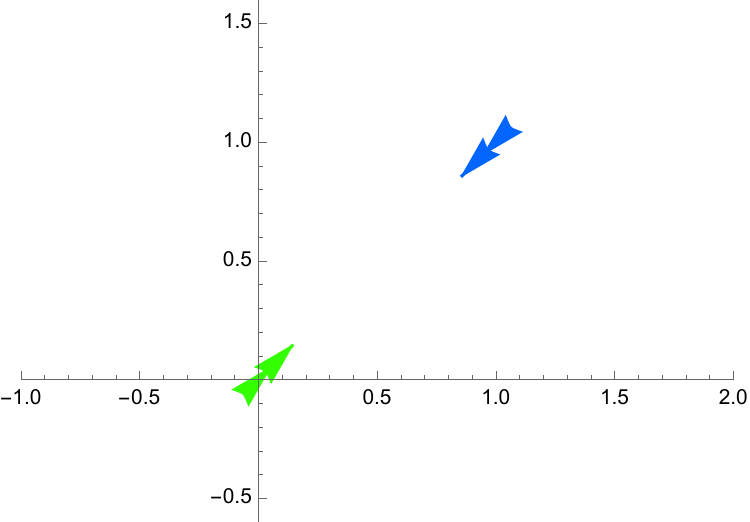}%
\caption{Ex. A.3: $\lambda_1=-1+\rho_2$,\\ $\lambda_2=1-\rho_2,\mu_1=\mu_2=0$.}

\end{subfigure}
\hfill
\begin{subfigure}{0.23\textwidth}
\includegraphics[width=\textwidth]{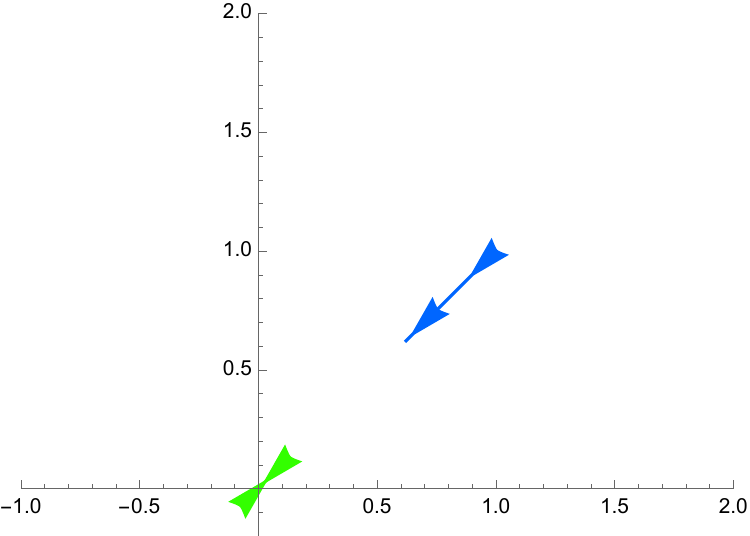}%
\caption{Example A.4: $\lambda_1=-1+\rho_2$,\\ $\lambda_2=0.5-\rho_2, \mu_1=\mu_2=0$.}

\end{subfigure}
\hfill
\begin{subfigure}{0.23\textwidth}
\includegraphics[width=\textwidth]{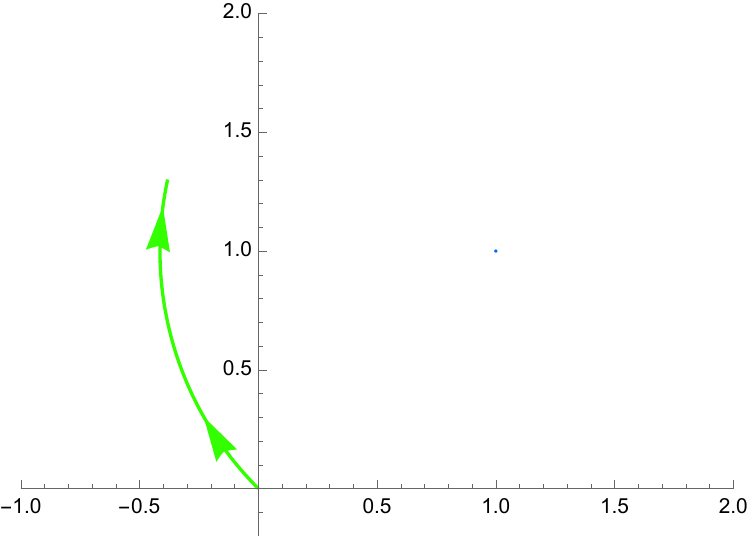}%
\caption{Example A.5: $\lambda_1=\lambda_2=0$,\\ $\mu_1=1,\mu_2=0$.}
\end{subfigure}
\hfill
\begin{subfigure}{0.23\textwidth}
\includegraphics[width=\textwidth]{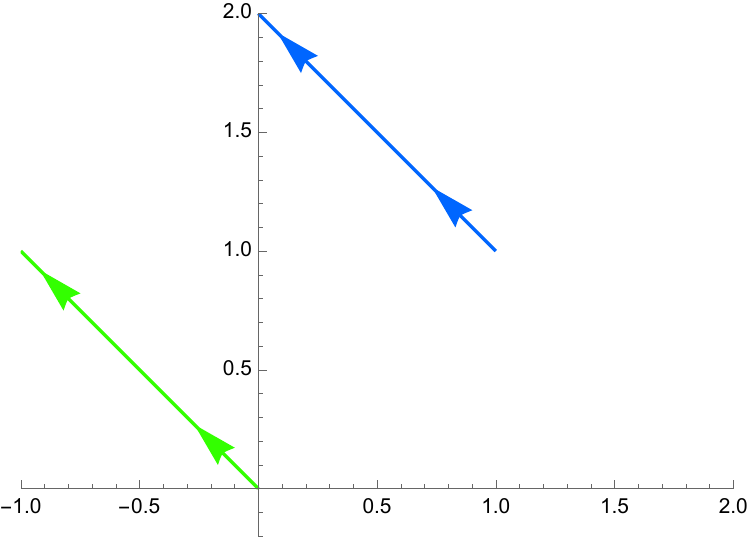}%
\caption{Example A.6: $\lambda_1=\lambda_2=0$,\\ $\mu_1=\mu_2=1$.}
\end{subfigure}
\hfill
\begin{subfigure}{0.23\textwidth}
\includegraphics[width=\textwidth]{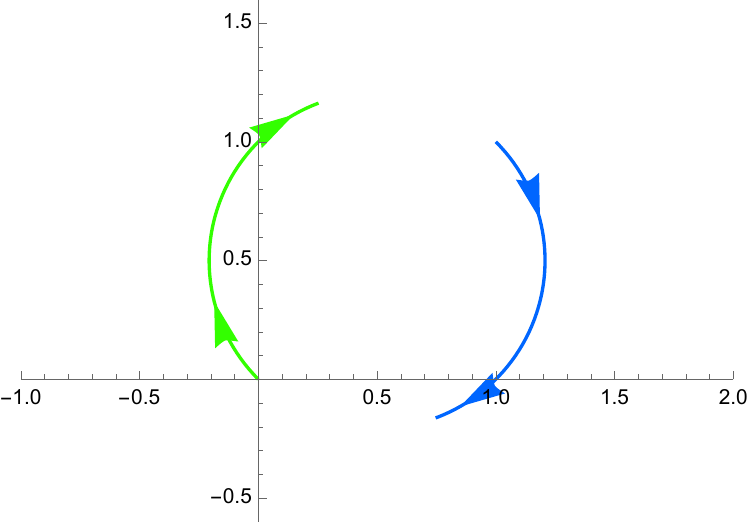}%
\caption{Example A.7: $\lambda_1=\lambda_2=0$,\\ $\mu_1=1,\mu_2=-1$.}
\end{subfigure}
\hfill
\begin{subfigure}{0.23\textwidth}
\includegraphics[width=\textwidth]{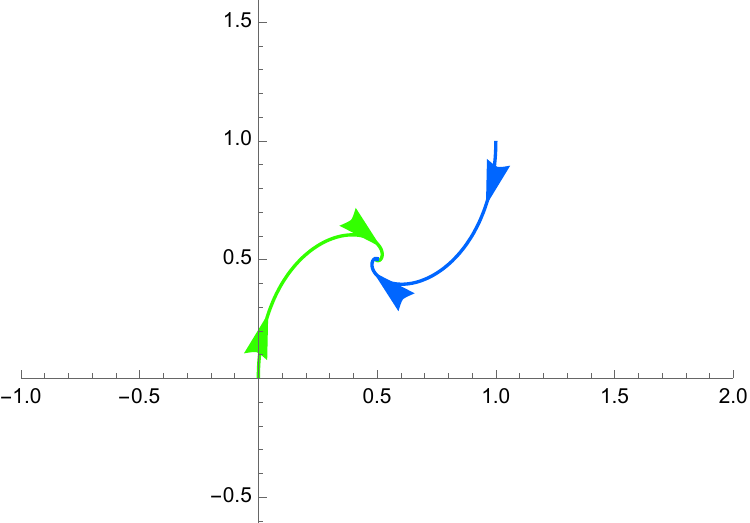}%
\caption{Example A.8: $\lambda_1=1,\lambda_2=-1$,\\ $\mu_1=1,\mu_2=-1$.}
\end{subfigure}
\hfill
\begin{subfigure}{0.23\textwidth}
\includegraphics[width=\textwidth]{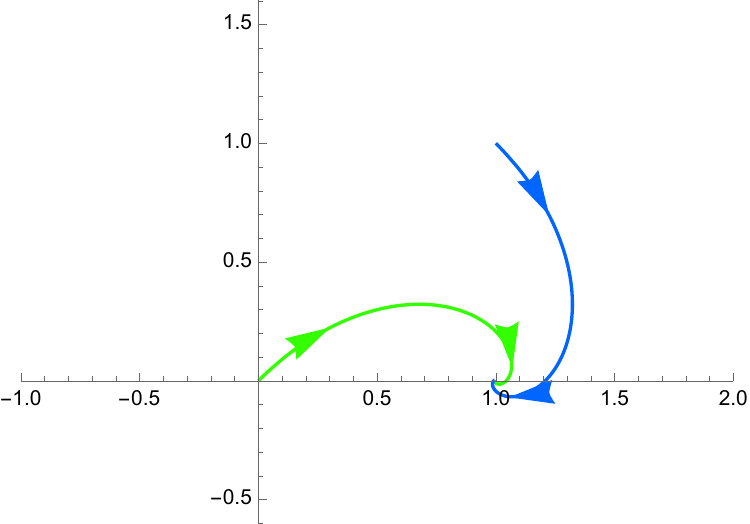}%
\caption{Example A.9: $\lambda_1=1,\lambda_2=0$,\\ $\mu_1=0,\mu_2=1$.}
\end{subfigure}
\hfill
\begin{subfigure}{0.23\textwidth}
\includegraphics[width=\textwidth]{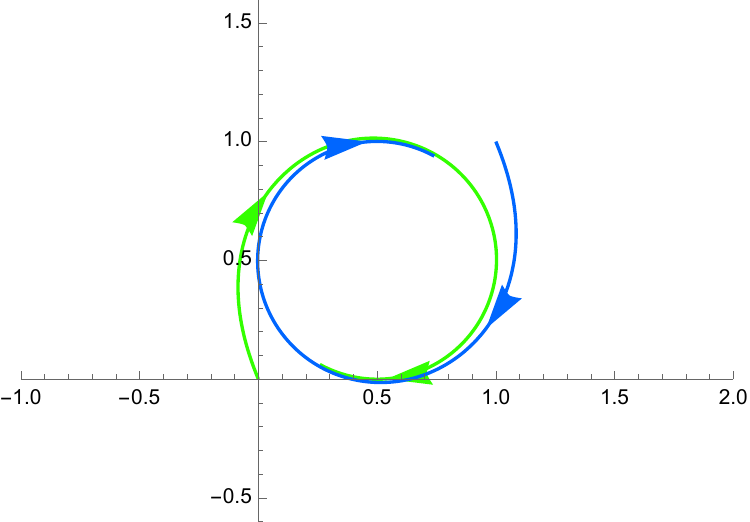}%
\caption{Ex. A.10: $\lambda_1=-1+\rho_2$,\\ $\lambda_2=1-\rho_2,\mu_1=1,\mu_2=-1$.}
\end{subfigure}
\hfill
\begin{subfigure}{0.23\textwidth}
\includegraphics[width=\textwidth]{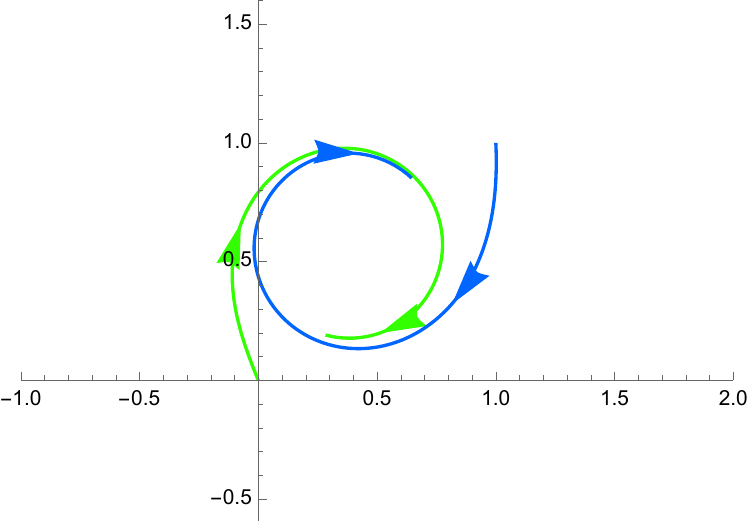}%
\caption{Ex. A.11: $\lambda_1=-1+\rho_2$,\\ $\lambda_2=0.5-\rho_2,\mu_1=1,\mu_2=-1$.}
\end{subfigure}
\hfill
\begin{subfigure}{0.23\textwidth}
\includegraphics[width=\textwidth]{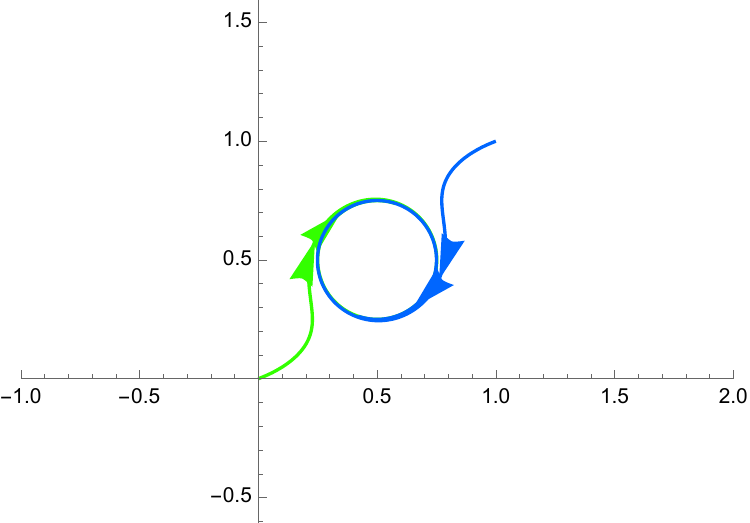}%
\caption{Ex. A.12: $\lambda_1=-0.5+\rho_2=-\lambda_2$, $\mu_1=1-\rho_2=-\mu_2$.}
\end{subfigure}
\hfill
\caption{Examples A (1 to 12).}
\label{f-A}
\end{figure}

\begin{remark} \label{r-Rn} The case of $M=\mathbb{R}^{n}$ and
$G=SO(n,\mathbb{R})\ltimes \R^n$ for $n>2$ is very different. Indeed, in this case, one has $\Centr=\R.\Id$, hence hypotheses of Corollary \ref{c-Z}, Item 1, are never satisfied. Then, one can resort to Corollary \ref{c-Z}, Item 2 if $N-1\geq n$ to describe FPD. As it has been already stated in the introduction, our results do not cover ``intermediate'' number of agents.
\end{remark}

\begin{remark} Contrarily to the case described in Remark \ref{r-Rn}, there is another interesting case in which assumptions of Corollary \ref{c-Z}, Item 1 hold: this is the group of unit quaternions $\mathcal{H}$ (or the semi-direct product $\mathcal{H}\ltimes \R^4$) acting on $\R^4$. Indeed, the Lie algebra $\mathcal{L}$ is generated by
the 3 quaternion matrices $i,j,k.$
\[
i=\left(
\begin{array}
[c]{cccc}%
0 & -1 & 0 & 0\\
1 & 0 & 0 & 0\\
0 & 0 & 0 & -1\\
0 & 0 & 1 & 0
\end{array}
\right)  ,\qquad
j=\left(
\begin{array}
[c]{cccc}%
0 & 0 & -1 & 0\\
0 & 0 & 0 & 1\\
1 & 0 & 0 & 0\\
0 & -1 & 0 & 0
\end{array}
\right)  ,\qquad
k=\left(
\begin{array}
[c]{cccc}%
0 & 0 & 0 & -1\\
0 & 0 & -1 & 0\\
0 & 1 & 0 & 0\\
1 & 0 & 0 & 0
\end{array}
\right) .
\]

The centralizer $\Centr$ is the full Lie algebra $\mathcal{H}^{\ast}$of
skew quaternions, $\mathcal{H}^{\ast}=\{a.Id+b.\hat{\imath}+c.\hat{\jmath
}+d.\hat{k}\}$ with 

\[
\hat{\imath}=\left(
\begin{array}
[c]{cccc}%
0 & -1 & 0 & 0\\
1 & 0 & 0 & 0\\
0 & 0 & 0 & 1\\
0 & 0 & -1 & 0
\end{array}
\right)  ,\qquad
\hat{\jmath}=\left(
\begin{array}
[c]{cccc}%
0 & 0 & 1 & 0\\
0 & 0 & 0 & 1\\
-1 & 0 & 0 & 0\\
0 & -1 & 0 & 0
\end{array}
\right)  ,\qquad
\hat{k}=\left(
\begin{array}
[c]{cccc}%
0 & 0 & 0 & -1\\
0 & 0 & 1 & 0\\
0 & -1 & 0 & 0\\
1 & 0 & 0 & 0
\end{array}
\right) .
\]

It is easy to check that the four vectors $x,\hat{\imath}.x,\hat{\jmath}.x,\hat
{k}.x$ are independent for all $x\in\R^4\setminus\{0\}$. Thus, Theorem \ref{mainth}, Item 1 or Corollary \ref{c-Z}, Item 1 can be used to characterize the FPD.
\end{remark}

\subsection{Gradient flows for 2 particles} In the case of $N=2$ particles, we are able to completely characterize the gradient flow PDs.
\begin{proposition}\label{p-GFSE2}
Let $F$ be a PD in the FPD associated to $(N=2,\mathbb{R}^{2},SO(2,\R)\ltimes
\mathbb{R}^{2})$, i.e. of the form \eqref{e-PD-SE2}. Then, it is a gradient flow if and only if $\lambda_1=-\lambda_2$ and $\mu_1=\mu_2=0$.
\end{proposition}
\begin{proof} It is easy to prove that both conditions $\lambda_1=-\lambda_2$ and $\mu_1=\mu_2=0$ imply that $F$ is a gradient flow. Indeed, define 
$$\phi(x_1,y_1,x_2,y_2):= \Phi(\sqrt{(x_1-x_2)^2+(y_1-y_2)^2})\mbox{~~with~~}\Phi(\rho)=\int_0^\rho \sqrt{s}\lambda_2(s)\,ds$$ and check that $F=\nabla \phi$.

We now prove the converse result. Let $\phi(x_1,x_2,x_2,y_2)$ be such that $\nabla \phi$ is a PD, i.e. of the form \eqref{e-PD-SE2}. A direct computation shows that, by imposing that $f_{1,1}=\partial_{x_1}\phi$ and $f_{2,1}=\partial_{x_2} \phi$, the condition $\partial_{x_2}f_{1,1}=\partial_{x_1}f_{2,1}$ ensures $\lambda_1=-\lambda_2$ and $\mu_1=-\mu_2$. We already know that the dynamics given by $\lambda_1=-\lambda_2$ with $\mu_1=\mu_2=0$ is a gradient flow, as proved above. By recalling that gradient flows form a vector space, we can then focus on the case $\lambda_1=\lambda_2=0$ and $\mu_1=-\mu_2$.

We now impose the condition $\partial_{y_2}f_{2,1}=\partial_{x_2}f_{2,2}$, that reads as
\begin{eqnarray*}
&& -\mu_2(\rho_2)+(y_1-y_2)\partial_{\rho_2} \mu_2 (\rho_2) \frac{y_2-y_1}{\rho_2}=\mu_2(\rho_2)+(x_2-x_1)\partial_{\rho_2} \mu_2 (\rho_2) \frac{x_2-x_1}{\rho_2},
\end{eqnarray*}
hence $2\mu_2(\rho_2)+\rho_2 \partial_{\rho_2}\mu_2(\rho_2)=0$. This implies $\mu_2(\rho_2)=\frac{k}{\rho_2^2}$. Remark that the domain is clearly given by $\rho_2\neq 0$, that is a $G$-invariant set. Again by recalling that gradient flows form a vector space, we see that we need only to check wether the following vector field is a gradient flow:

\begin{equation}\label{e-mu2}
\dot x_1=\frac{y_2-y_1}{\rho_2^2},\qquad \dot y_1=\frac{x_1-x_2}{\rho_2^2}, \qquad \dot x_2=\frac{y_1-y_2}{\rho_2^2}, \qquad \dot y_2=\frac{x_2-x_1}{\rho_2^2}.
\end{equation}

By duality in $\R^4$, it is a gradient flow if and only if the dual 1-form 
$$\omega=\frac{(y_1-y_2) d(x_2-x_1)+(x_2-x_1)d(y_2-y_1)}{\rho_2^2}$$
 is exact. By writing it in coordinates $(\rho_2,\theta_2,x_2,y_2)$ with $\rho_2(\cos(\theta_2),\sin(\theta_2))=(x_2-x_1,y_2-y_1)$, we have $\omega=d\theta_2$. It is clear that this form is not exact. Thus, the vector field \eqref{e-mu2} is not a gradient flow. This  in turn implies that a gradient flow that is a PD necessarily satisfies $\mu_1=\mu_2=0$.
\end{proof}
\begin{remark} The proof above shows that a complete characterization of gradient flow PDs is difficult, since one needs to check wether forms are exact on domains that are not simply connected. For this reason, we have no complete characterization for the case $N>2$, for which we only have the partial results given by Proposition \ref{p-GF}.
\end{remark}

\subsection{Permutation-equivariant population dynamics\label{homdep}} We now add permutation equivariance to the previous setting, by studying FPEPD introduced in Definition \ref{d-PEPD}. Since we already observed that the centralizer $\Centr$ of $SO(2,\R)$ has dimension $n=2$, we can apply Corollary \ref{c-trasl-PEPD} to compute the FPEPD in this case. We have the following result.
\begin{proposition} Consider the FPEPD associated to $(N,\mathbb{R}^{2},SO(2,\R)\ltimes
\mathbb{R}^{2})$ with $N>1$.

Consider the dynamics taking place on the open dense set 
\begin{eqnarray*}
\mathcal{O}&=&\Big\{(X_1,\ldots,X_N)\in\R^{2N} \mbox{~~such that~~}\sum_{j\neq i}(X_j- X_i) \neq (0,0)\mbox{~for all~}i=1,\ldots,N\Big\}.
\end{eqnarray*}
Write in polar coordinates $X_j-X_i=\rho_{j,i}(\cos(\theta_{j,i}),\sin(\theta_{j,i}))$ and define $\tau_{j,i}:=\theta_{j,i}-\theta_{i+1,i}$, with the convention that $i=N+1$ is replaced by $i=1$ and that the difference of angles is mod-$2\pi$.

Then, the FPEPD is composed of all vector fields of the form 
$$F=F_{1}%
+\ldots+F_{N}\qquad \mbox{~~with~~}\qquad F_{i}=f_{i,1}\partial_{x_{i}}+f_{i,2}%
\partial_{y_{i}},$$ and 
$$\left(
\begin{array}
[c]{c}%
f_{i,1}\\
f_{i,2}%
\end{array}
\right)  =\left(%
\begin{array}
[c]{cc}%
\sum_{j\neq i}(x_{j}- x_{i}) & -\sum_{j\neq i}(y_{j}-y_i)\\
\sum_{j\neq i}(y_{j}-y_{i}) & \sum_{j\neq i}(x_{j}-x_{i})%
\end{array}
\right)\binom{\lambda(\check{\rho}_i,\check{\tau_i})}{\mu(\check{\rho}_i,\check{\tau_i})},$$
where functions $\lambda,\mu$ are arbitrary analytic functions that are invariant with respect to permutations of both the first $N-1$ coordinates
$$\check{\rho}_i:=(\rho_{1,i},\ldots,\rho_{i-1,i},\rho_{i+1,i},\ldots,\rho_{N,i})$$
and the remaining $N-2$ coordinates 
$$\check{\tau}_i:=(\tau_{1,i},\ldots,\tau_{i-1,i},\tau_{i+2,i},\ldots,\tau_{N,i}).$$
\end{proposition}
\begin{proof} The proof is a direct application of Corollary \ref{c-trasl-PEPD}. The only remarkable detail is that all functions that are jointly radial (i.e. invariant by rototranslation) and permutation invariant are of the form $\lambda,\mu$ given in the statement; this fact has already been proved for Proposition \ref{p-SE2}.
\end{proof}
\begin{remark} This result, combined with Proposition \ref{p-GFSE2}, shows that the only possible PDs for $N=2$ that are gradient flows are also PEPDs.\end{remark}

We do not provide additional examples. For $N=2$ particles, we already had PEPD in Examples A.(1-3-7-8-10-12) in Figure \ref{f-A}. Examples for $N>2$ particles are similar, again with $\lambda>0$ promoting convergence and $\mu>0$ promoting clockwise rotation.

\section{Population Dynamics on the relativistic spaces}\label{relati}

In this section, we focus on Population Dynamics on the relativistic spaces.

\subsection{Population Dynamics on the relativistic line} \label{s-relat2}

In this section, we deal with PDs on the relativistic line: the state space is the pseudo-Euclidean plane $\R^2$ endowed with the quadratic form with signature $(1,1)$. We denote by $c>0$ the speed of light and consider the following scalar product: $$(T_1,x_1)\cdot (T_2,x_2)=c^2 T_1T_2-x_1x_2.$$

The group is the semidirect product $G=SO(1,1)\ltimes \R^2$ that preserves the associated quadratic form. We have the following result.

\begin{proposition} \label{p-relat2} Consider the FPD associated to $(N,\mathbb{R}^{2},SO(1,1)\ltimes \mathbb{R}^{2})$. Then:
\begin{itemize}
\item for $N=1$, the FPD contains the zero vector field only;

\item for $N=2$, all PDs are given by $F=F_{1}+F_2$ with $F_{i}=f_{i,1}\partial_{T_{i}}+f_{i,2}\partial_{x_{i}}$, where
\begin{equation*}
\left(
\begin{array}
[c]{c}%
f_{i,1}\\
f_{i,2}%
\end{array}
\right) =\left(
\begin{array}
[c]{cc}%
T_2-T_1 & c^{-1}(x_2-x_1)\\
x_2-x_1 & c(T_2-T_1)
\end{array}
\right)\left(
\begin{array}{c}\phi_i(r)\\\psi_i(r)\end{array}\right) \label{ratata}%
\end{equation*}

with $\phi_i(r),\psi_i(r)$ analytic functions of $r=\sqrt{c^2(T_2-T_1)^{2}-(x_2-x_1)^{2}}$;

\item for $N\geq 3$, and choosing any two distinct indexes $k_1,k_2\in \{2,\ldots, N\}$, all PDs are of the form 
$$F(X)=\sum_{i=1}^N f_i(Z)\partial_{X_i}\qquad\mbox{~~ with ~~}\qquad f_i=\psi_{i,1}(Z)Z_{k_1}+\psi_{i,2}(Z)Z_{k_2},$$
 where $\psi_{i,k}(Z)$ are $\hat{\LL}$-jointly radial functions of $Z=(Z_2,\ldots,Z_N)$ with $Z_i:=(T_i-T_1,x_i-x_1)$.

\end{itemize}
\end{proposition}
\begin{proof} For $N=1$, we have that Proposition \ref{p-1} ensures that PDs are contained in the set of constant vector fields. By adding equivariance with respect to $SO(1,1)$, we have that the only PD is the zero vector field.

The case $N\geq 3$ is a direct application of Corollary \ref{c-Z}, Item 2. One can also apply Corollary \ref{c-Z}, Item 1.

The case $N= 2$ is a direct application of Corollary \ref{c-Z}, Item 1. Indeed, the centralizer of $so(1,1)=\R.K$ with 
$$K=\left(\begin{array}{cc} 0 & c^{-1}\\
c & 0\end{array}\right)$$ is $\Centr=\mathbb{R}.\Id+\R.K$, that has dimension $n=2$. Then, solutions are of the form $\phi(Z)Z+\psi(Z) K Z$, where $Z=(T_2-T_1,x_2-x_1)$. Moreover, $\phi(Z),\psi(Z)$ are $so(1,1)$-radial, thus depending on  $r=\|Z\|=\sqrt{c^2(T_2-T_1)^{2}-(x_2-x_1)^{2}}$ only.

\end{proof}

\begin{remark}\label{r-Eproper} The case $N=2$ is an interesting example in which it holds $\EE\subsetneq\SS$, i.e. in which elementary solutions do not provide all solutions. We can anyway provide a complete classification, due to the equality $ \CC =\SS$, i.e. observing that all solutions are solutions given by the centralizer.
\end{remark}

\begin{remark}\label{r-r-dyn} A direct computation for $N=2$ shows that $\dot r=r(\phi_1(r)-\phi_2(r))$. In particular, this implies that terms $\psi_1,\psi_2$ do not contribute to the dynamics of $r$ (i.e. the dynamics on the quotient space).\end{remark}

We now provide some simple examples for $N=2$. In Examples B.1-2-3, we consider the case of $\phi_1(r)=r$, $\phi_2(r)=1$ and $\psi_1=\psi_2=0$ with increasing values of $c\in\{1,10,100\}$. By Remark \ref{r-r-dyn}, it holds $\partial_t r(t)=r(1-r)$, hence the variable $r$ converges to the value 1. The limit dynamics is then given by $\phi_1=\phi_2=1$, i.e. translations with constant velocity 1 in all variables $T_1,x_1,T_2,x_2$. This corresponds to the fact that the relativistic dynamics converges to the so-called {\it relative equilibria} of the classical dynamics, i.e. classical dynamics with constant $x_1-x_2$. The key remark here is that the convergence of the relativistic dynamics to the classical one is faster when values of $c$ increase, as expected.

In Examples B.4-5-6, we investigate the role of $\psi_1,\psi_2$. They do not contribute to the dynamics of the variable $r$, but we show that they play a strong role in the dynamics of the state variable. In particular, we show that the dynamics is qualitatively different for different values of the parameter $c\in\{1,10,100\}$. We choose again $\phi_1(r)=r$, $\phi_2(r)=1$, ensuring convergence of $r$ to 1. By defining $\psi_1=100/c=-\psi_2$, we find that:
\begin{itemize}
\item for $c=1$ the distances $T_2-T_1,x_2-x_1$ explode;
\item for $c=10$ the distance $T_2-T_1$ converges to 0, while $x_2-x_1$ converges to a constant (i.e. to a classical relative equilibrium);
\item for $c=100$ the dynamics is similar to the case $c=10$, but convergence is faster for both variables.
\end{itemize}

\begin{figure}[htb]
\centering
\begin{subfigure}{0.25\textwidth}
\includegraphics[width=\textwidth]{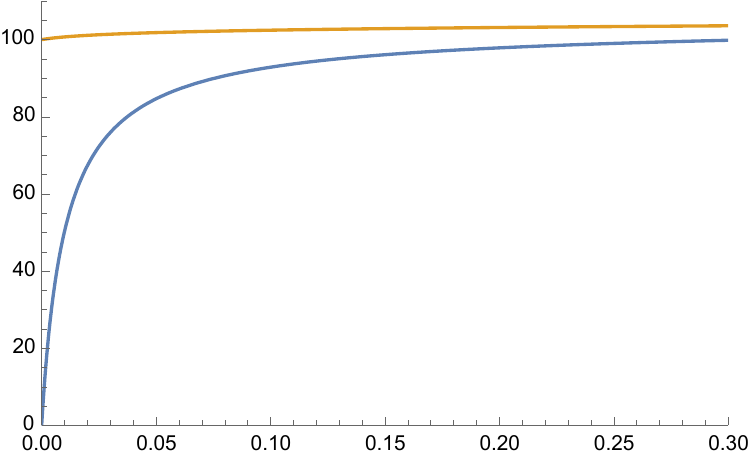}%
\end{subfigure}
\hfill
\begin{subfigure}{0.25\textwidth}
\includegraphics[width=\textwidth]{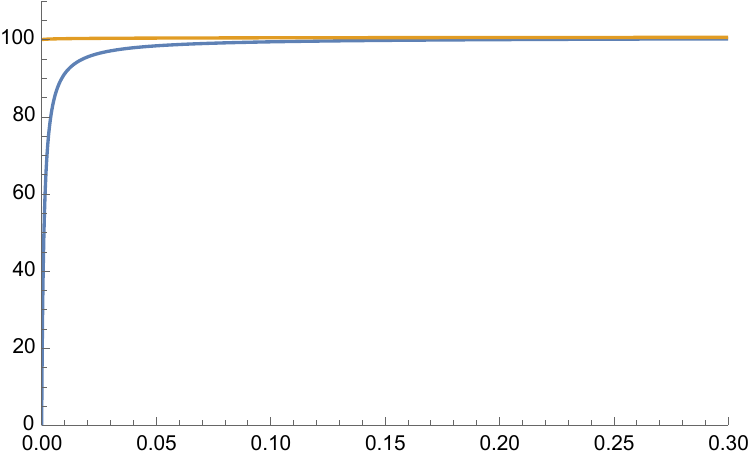}%
\end{subfigure}
\hfill
\begin{subfigure}{0.25\textwidth}
\includegraphics[width=\textwidth]{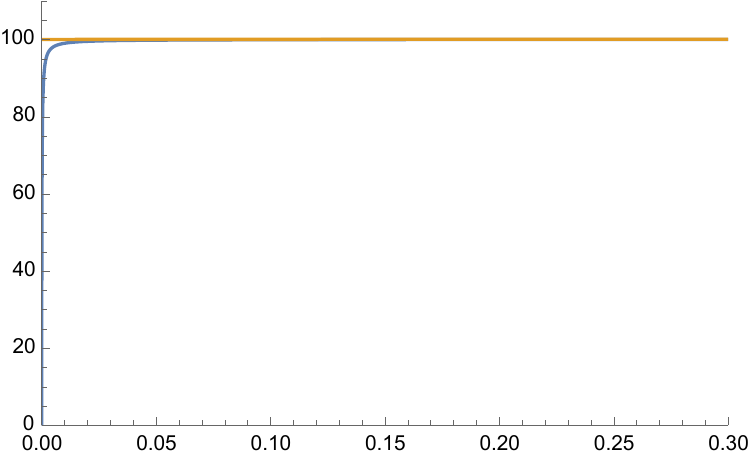}%
\end{subfigure}
\hfill
\begin{subfigure}{0.25\textwidth}
\includegraphics[width=\textwidth]{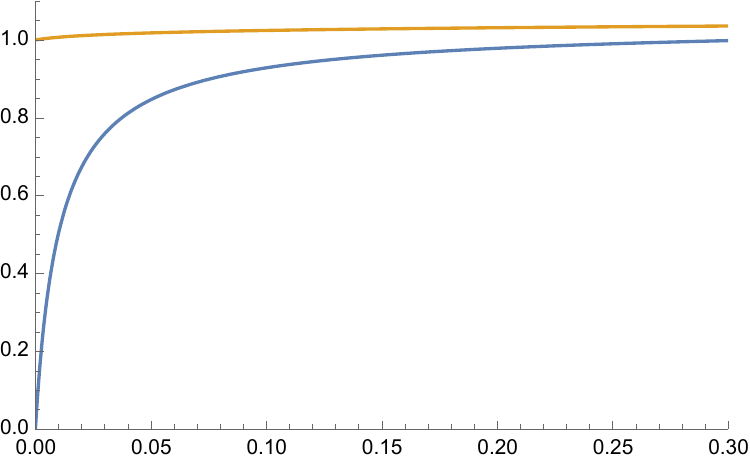}%
\caption{Example B.1 with $c=1$.}
\end{subfigure}
\hfill
\begin{subfigure}{0.25\textwidth}
\includegraphics[width=\textwidth]{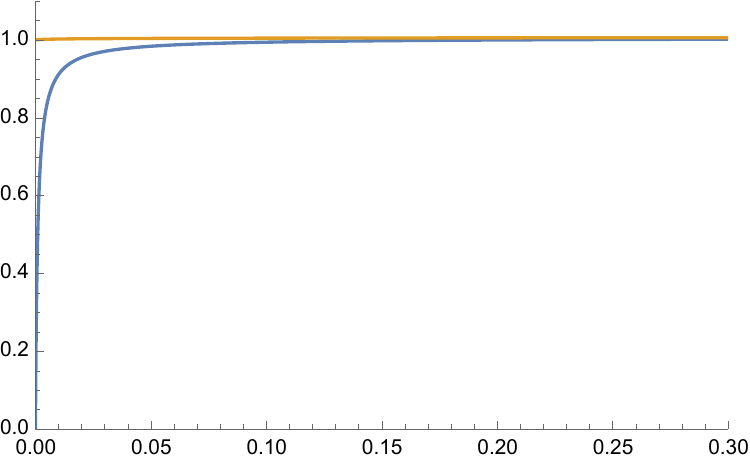}%
\caption{Example B.2 with $c=10$.}
\end{subfigure}
\hfill
\begin{subfigure}{0.25\textwidth}
\includegraphics[width=\textwidth]{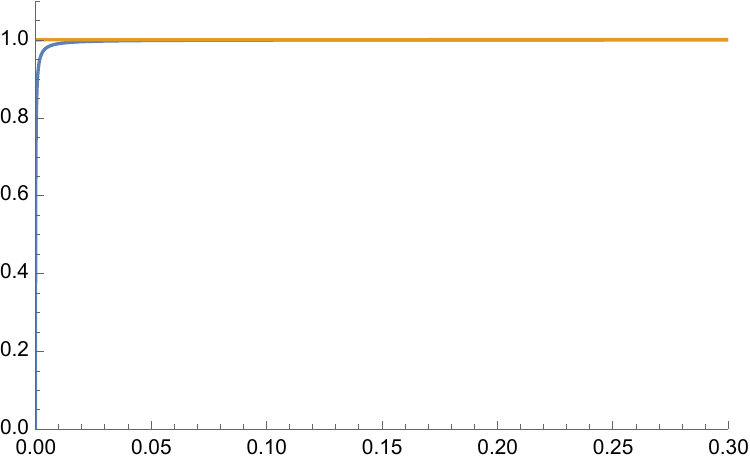}%
\caption{Example B.3 with $c=100$.}
\end{subfigure}
\hfill
\begin{subfigure}{0.25\textwidth}
\includegraphics[width=\textwidth]{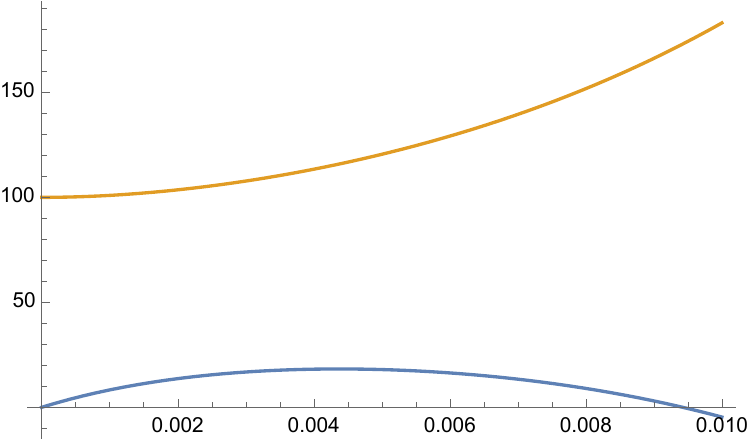}%
\end{subfigure}
\hfill
\begin{subfigure}{0.25\textwidth}
\includegraphics[width=\textwidth]{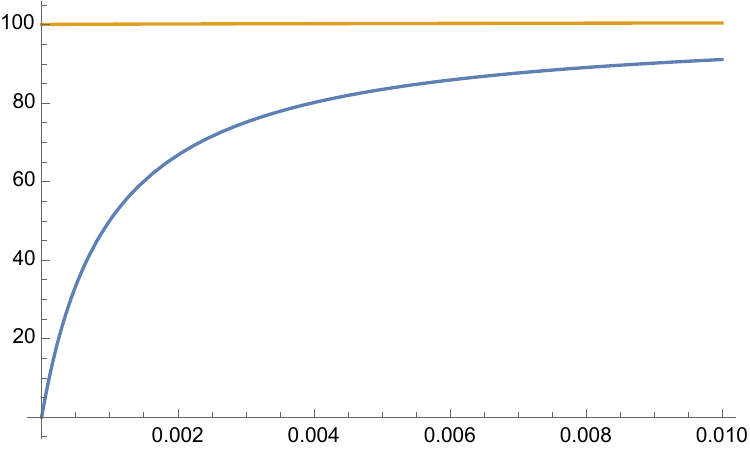}%
\end{subfigure}
\hfill
\begin{subfigure}{0.25\textwidth}
\includegraphics[width=\textwidth]{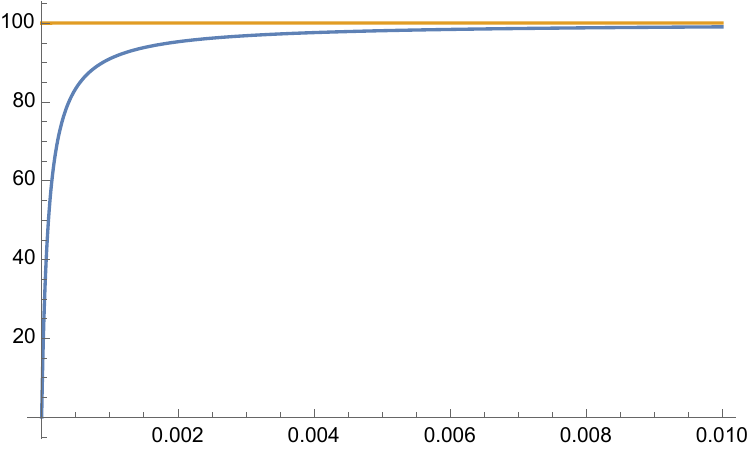}%
\end{subfigure}
\hfill
\begin{subfigure}{0.25\textwidth}
\includegraphics[width=\textwidth]{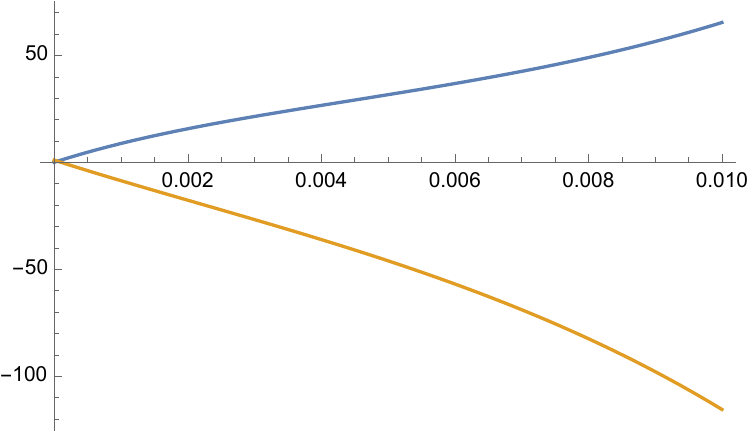}%
\caption{Example B.4 with $c=1$.}
\end{subfigure}
\hfill
\begin{subfigure}{0.25\textwidth}
\includegraphics[width=\textwidth]{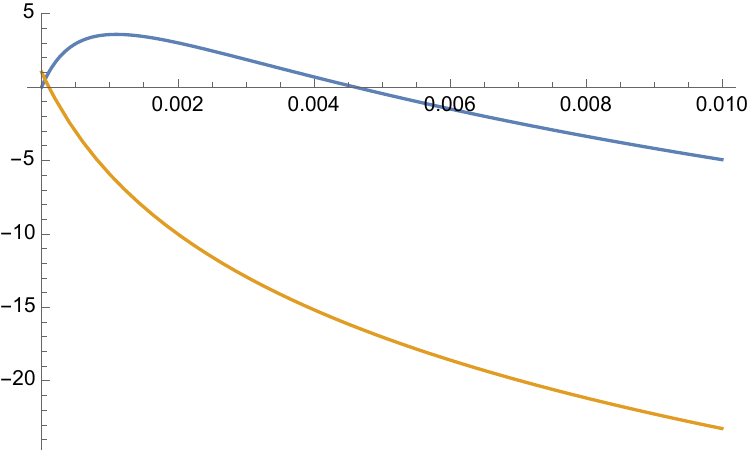}%
\caption{Example B.5 with $c=10$.}
\end{subfigure}
\hfill
\begin{subfigure}{0.25\textwidth}
\includegraphics[width=\textwidth]{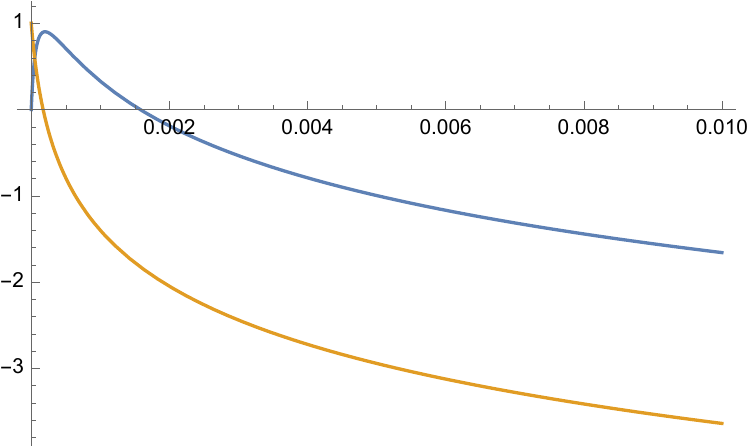}%
\caption{Example B.6 with $c=100$.}
\end{subfigure}
\hfill
\caption{Examples B (1 to 6). Variables $T$ above, $x$ below; Variables 1 (blue) and 2 (orange).}
\label{f-B}
\end{figure}

\subsection{Population Dynamics on the relativistic plane}\label{s-relat3}

In this section, we deal with PDs on the relativistic plane: the state space is the pseudo-Euclidean space $\R^3$ endowed with the quadratic form with signature $(1,2)$. Given two states $(T_1,x_1,y_1)$, $(T_2,x_2,y_2)$, the scalar product is 
$$(T_1,x_1,y_1)\cdot(T_2,x_2,y_2)=T_1T_2-x_1x_2-y_1y_2.$$

The group is then the semidirect product $G=SO(1,2)\ltimes \R^3$. We treat the cases $N=1$ and $N\geq 4$ with the general theory. We are then left with two special dimensions, that are $N=2,3$. For $N=2$, we prove that PDs are determined by elementary solutions via a direct computation. Instead, for $N=3$ we have no general statement. The results are summarized as follows.

\begin{proposition} \label{p-relat} Consider the FPD associated to $(N,\mathbb{R}^{3},SO(1,2)\ltimes \mathbb{R}^{3})$. Then:

\begin{itemize}
\item for $N=1$, the FPD is composed of the zero vector field only;
\item for $N=2$, all PDs are of the form
\begin{eqnarray*}
F(X_1,X_2)&=&\psi_1(r) \left((T_2-T_1)\partial_{T_1}+
(x_2-x_1)\partial_{x_1}+
(y_2-y_1)\partial_{y_1}\right)+\\
&&\psi_2( r) \left((T_2-T_1)\partial_{T_2}+
(x_2-x_1)\partial_{x_2}+
(y_2-y_1)\partial_{y_2}\right),
\end{eqnarray*}

with $r=\sqrt{(T_2-T_1)^2-(x_2-x_1)^2-(y_2-y_1)^2}$;
\item for $N=3$, the FPD {\bf contains} vector fields of the form 
\begin{eqnarray*}
F(X_1,X_2,X_3)&=&\sum_{i=1}^3
\phi_{i}(Z_2,Z_3) \left((T_2-T_1)\partial_{T_i}+
(x_2-x_1)\partial_{x_i}+
(y_2-y_1)\partial_{y_i}\right)+\\
&&\psi_{i}(Z_2,Z_3) \left((T_3-T_1)\partial_{T_i}+
(x_3-x_1)\partial_{x_i}+
(y_3-y_1)\partial_{y_i}\right),
\end{eqnarray*}

with 
$$Z_2:=(T_2-T_1,x_2-x_1,y_2-y_1),\qquad Z_3:=(T_3-T_1,x_3-x_1,y_3-y_1)$$
where $\phi_{i},\psi_{i}$ are analytic and $\hat{\LL}$-jointly radial with respect to the action of $SO(1,2)$;

\item for $N\geq 4$, and choosing any three distinct indexes $k_1,k_2,k_3\in \{1,\ldots, N\}$, all PDs are of the form 
$$F(X)=\sum_{i=1}^N f_i(Z)\partial_{X_i}\qquad\mbox{~~ with ~~}\qquad f_i=\psi_{i,1}(Z)Z_{k_1}+\psi_{i,2}(Z)Z_{k_2}+\psi_{i,3}(Z)Z_{k_3},$$
 where $\psi_{i,k}(Z)$ are $\hat{\LL}$-jointly radial functions of $Z=(Z_2,\ldots,Z_N)$ with $Z_i:=(T_i-T_1,x_i-x_1,y_i-y_1)$.

\end{itemize}
\end{proposition}
\begin{proof} The cases $N=1$ and $N\geq 4$ are identical to Proposition \ref{p-relat2}.

We now study $N= 2$, for which we introduce the difference variable $$Z=(T,x,y):=(T_2-T_1,x_2-x_1,y_2-y_1)$$ and consider vector fields of a single particle of coordinate $Z$ that are equivariant under the action of $SO(1,2)$, i.e. solutions of \eqref{e-PDEz}. For this reason, we introduce polar coordinates in the $(x,y)$ difference variables, i.e.
$$(x_2-x_1,y_2-y_1)=\rho(\cos(\theta),\sin(\theta))\mbox{~~~with~}\rho\geq 0, \theta\in[0,2\pi).$$

It is easy to prove that the Lie algebra $so(1,2)$ is generated by 
$$V_1=x\partial_T+T\partial_x,\qquad V_2=y\partial_T+T\partial_y,\qquad V_3=-y\partial_x+x\partial_y.$$
By rewriting them in cylindrical coordinates $(T,\rho,\theta)$, it holds
$$V_1=\rho\cos(\theta)\partial_T+T\cos(\theta)\partial_\rho -T\frac{\sin(\theta)}\rho \partial_\theta, \qquad V_2=\rho\sin(\theta)\partial_T+T\sin(\theta)\partial_\rho + T\frac{\cos(\theta)}\rho \partial_\theta, \qquad V_3=\partial_\theta.$$

Let $$F=f(T,\rho,\theta)\partial_T+g(T,\rho,\theta)\partial_\rho+h(T,\rho,\theta)\partial_\theta$$ be a PD. The condition $[F,V_3]=0$ is equivalent to require that $f,g,h$ do not depend on $\theta$. For simplicity of notation, we now drop the dependence of $f,g,h$ on the remaining variables $T,\rho$. Then, we are left with imposing $[F,{V_1}]=[F,{V_2}]=0$, i.e.
\begin{eqnarray*}
[F,V_1]&=&f \left(\cos(\theta)\partial_\rho-\frac{\sin(\theta)}\rho\partial_\theta\right)+g \left(\cos(\theta)\partial_T+T\frac{\sin(\theta)}{\rho^2}\partial_\theta\right)-h V_2\\
&&-\rho\cos(\theta) ((\partial_T f)\partial_T+(\partial_T g) \partial_\rho+(\partial_T h)\partial_\theta)
-T\cos(\theta)((\partial_\rho f)\partial_T+(\partial_\rho g) \partial_\rho+(\partial_\rho h)\partial_\theta)=0,
\end{eqnarray*}
\begin{eqnarray*}
[F,V_2]&=&f \left(\sin(\theta)\partial_\rho+\frac{\cos(\theta)}\rho\partial_\theta\right)+g \left(\sin(\theta)\partial_T-T\frac{\cos(\theta)}{\rho^2}\partial_\theta\right)+h V_1\\
&&-\rho\sin(\theta) ((\partial_T f)\partial_T+(\partial_T g) \partial_\rho+(\partial_T h)\partial_\theta)
-T\sin(\theta)((\partial_\rho f)\partial_T+(\partial_\rho g) \partial_\rho+(\partial_\rho h)\partial_\theta)=0.\\
\end{eqnarray*}

It is clear that this condition is equivalent to $$\cos(\theta) [F,{V_1}]+\sin(\theta) [F,{V_2}]=0,\qquad -\sin(\theta) [F,{V_1}]+\cos(\theta)[F,{V_2}]=0,$$
i.e.
\begin{eqnarray}
&&f \partial_\rho+g \partial_T-h \frac{T}{\rho}\partial_\theta-\rho ((\partial_T f)\partial_T+(\partial_T g) \partial_\rho+(\partial_T h)\partial_\theta)-T((\partial_\rho f)\partial_T+(\partial_\rho g) \partial_\rho+(\partial_\rho h)\partial_\theta)=0,\label{e-relat1}\\
&&f \frac{1}\rho \partial_\theta-g  \frac{T}{\rho^2}\partial_\theta+h (\rho\partial_T+T\partial_\rho)=0.\label{e-relat2}
\end{eqnarray}
Equation \eqref{e-relat2} is equivalent to both 
\begin{equation}
h=0,\qquad \rho f =T g .\label{e-relat3}
\end{equation}
By plugging these conditions into \eqref{e-relat1}, we have
$$f\partial_\rho+g\partial_T-(\partial_T (Tg))\partial_T-\rho(\partial_Tg)\partial_\rho -T(\partial_\rho f)\partial_T-(\partial_\rho (\rho f))\partial_\rho=0.$$
A direct computation shows that this condition is equivalent to 
\begin{equation}
\partial_T g +\partial_\rho f=0.\label{e-relat4}
\end{equation} By using \eqref{e-relat3}, we write  $f=T \psi, g=\rho\psi$ for some $\psi=\psi(T,\rho)$ to be found. Equation \eqref{e-relat4} then reads as $\rho\partial_T\psi+T\partial_\rho\psi=0$, that is equivalent to state that $\psi$ only depends on the variable $r:=\sqrt{T^2-\rho^2}$. Summing up, it holds 
\begin{equation}
F(Z)=\psi(r)(T\partial_T+\rho\partial_\rho)
\label{e-F-SO12}
\end{equation}
for some $\psi(r)$. 
Going back to cartesian coordinates and to the original dynamics for $N=2$ particles, we have the result.

For $N=3$, we write the difference variables $Z_2,Z_3$ as in the statement. We then observe that elementary solutions $f$ of \eqref{e-PDEz} are of the form $\phi(Z_2,Z_3)Z_2+\psi(Z_2,Z_3)Z_3$ with $\phi,\psi$ being $\hat{\LL}$-jointly radial with respect to the action of $SO(1,2)$. The statement then follows from the fact that elementary solutions are solution, see Proposition \ref{mainprop}.
\end{proof}

\begin{remark}\label{r-proper-inclusion-2}
The case $N=2$ is interesting, since the following inclusions of solutions holds $\EE\subsetneq\CC=\SS$, as already discussed in Remark \ref{r-Eproper}. In the case case $N=3$, we are unable to completely describe the FPD, since we only have $\EE=\CC\subset \SS$, but it is unclear wether the last inclusion is proper or not.

A very similar case is given by $SO(3,\R)\ltimes \R^3$ acting on $\R^3$, already discussed in Remark \ref{r-Rn}. There, the centralizer is too small to describe all solutions via Corollary \ref{c-Z}, Item 1. Anyway, computations for $N=1$ or $N=2$ can be carried out exactly as in the proof of Proposition \ref{p-relat}. This is one more example of the fact that one can characterize PDs either for small $N$ by computations, or for large $N$ by Corollary \ref{c-Z}, Item 2. The intermediate values of $N$ are not completely solved by the theory developed here.
\end{remark}

We do not provide examples, since the case $N=1$ is trivial and $N\geq 2$ is hard to picture, since one needs a plot in dimension 3 or larger. If one considers PDs only depending on $T_i$ and $\rho_i=\|(x_i-x_1,y_i-y_1)\|$, i.e. not depending on the argument $\arg(x_i-x_1,y_i-y_1)$, observe that the dynamics of the $T_i,\rho_i$ variables is given by PDs on the relativistic line, studied in Section \ref{s-relat2}.

\section{Population Dynamics on spheres}
\label{s-sphere}

In this section, we study Population Dynamics on the spheres of dimension 1 and 2.

\subsection{Population dynamics on the circle $S^1$}\label{s-S1}

In this section, we study Population Dynamics on the circle $S^1\subset \R^2$. The natural group of invariance is $SO(2,\R)$, the rotations of the plane $\R^2$ around the origin, that is isomorphic to $S^1$. We are then in the case in which the Lie group $G$ and the manifold $M$ on which agents evolve coincide. We treated this case in Section \ref{s-GM}.

\begin{proposition} Consider the FPD associated to $(N,S^1,SO(2,\R))$. Denote by $\theta_i$ the angular position of agent $i\in\{1,\ldots,N\}$.  Then, 
\begin{itemize}
\item for $N=1$, all PDs are given by constant vector fields;
\item for $N\geq 2$ all PDs are given by $F(\theta_1,\ldots,\theta_N)=\sum f_i\partial_{\theta_i}$, where 
\begin{equation}\label{e-fi-group}
f_i(\theta_1,\ldots,\theta_N)=\phi_i(\theta_2-\theta_1,\theta_3-\theta_1,\ldots,\theta_N-\theta_1)
\end{equation}
for some $\phi_i(\alpha_2,\ldots,\alpha_N)$ analytic function.
\end{itemize}
\end{proposition}
\begin{proof} The proof is a direct application of Proposition \ref{p-LieGroups}. Recall that $\alpha\in SO(2)$ acts on $\theta\in S^1$ by addition $\theta+\alpha$, thus Proposition \ref{p-LieGroups}, Statement 1 ensures that jointly equivariant functions $f_i$ are of the form given in \eqref{e-fi-group}; in particular, for $N=1$ they are constant. Since the group operation is commutative, the basis of both right- and left-invariant vector fields is simply $\partial_\theta$. By Proposition \ref{p-LieGroups}, Statement 2, we have the result.
\end{proof}

We provide some examples for $N=2$. In Figure \ref{f-S1}, we show the time evolution of the states, that are in the interval $[0,2\pi]$ with identification $0=2\pi$. Remark that one needs to have functions $\phi_i$ that are $2\pi$-periodic with respect to their argument, that is always $\theta_2-\theta_1$.

\begin{description}
\item[Example C.1] The initial state is $\theta_1(0)=\pi-0.1$ and $\theta_2(0)=0.1$. By choosing $\phi_1(a)=-\sin(a)+0.1$,  $\phi_2(a)=-\sin(a)+0.1$, we have convergence of both particles towards a trajectory of constant velocity 0.1.
\item[Example C.2] The initial state is $\theta_1(0)=-0.1$ and $\theta_2(0)=0.1$. By choosing $\phi_1(a)=|\sin(a/2)|$,  $\phi_2(a)=-\phi_1(a)$, we have initial divergence of particles. Yet, particles then converge to the same point $\pi=-\pi$.
\item[Example C.3] The initial state is again $\theta_1(0)=-0.1$ and $\theta_2(0)=0.1$. By choosing $\phi_1(a)=\sin(a)$,  $\phi_2(a)=-\phi_1(a)$, we have divergence of particles. Compactness of the manifold implies that the distance between particles cannot go to infinity, but to its maximum value $\pi$.
\item[Example C.4] The initial state is $\theta_1(0)=0$ and $\theta_2(0)=\pi$. By choosing $\phi_1(a)=\sin(2a)+2$,  $\phi_2(a)=\sin(a)$, in particular with different frequencies, we have a more complex behavior.
\end{description}

\begin{figure}[htb]
\centering
\begin{subfigure}{0.48\textwidth}
\includegraphics[width=\textwidth]{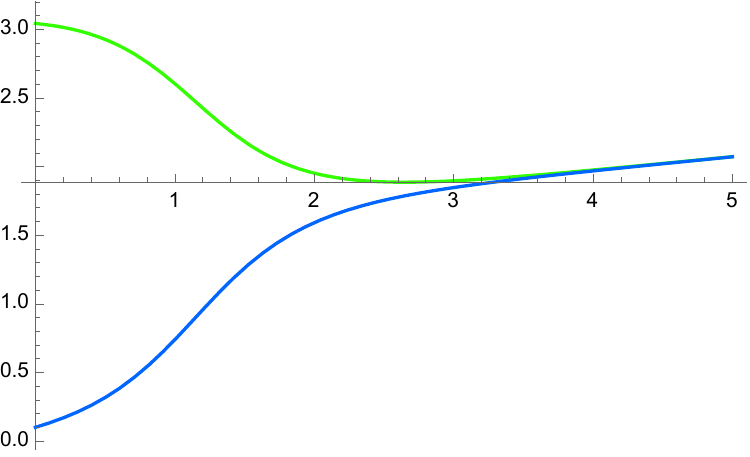}%
\caption{Example C.1: $\phi_1(a)=\phi_2(a)=-\sin(a)+0.1$.}
\end{subfigure}
\hfill
\begin{subfigure}{0.48\textwidth}
\includegraphics[width=\textwidth]{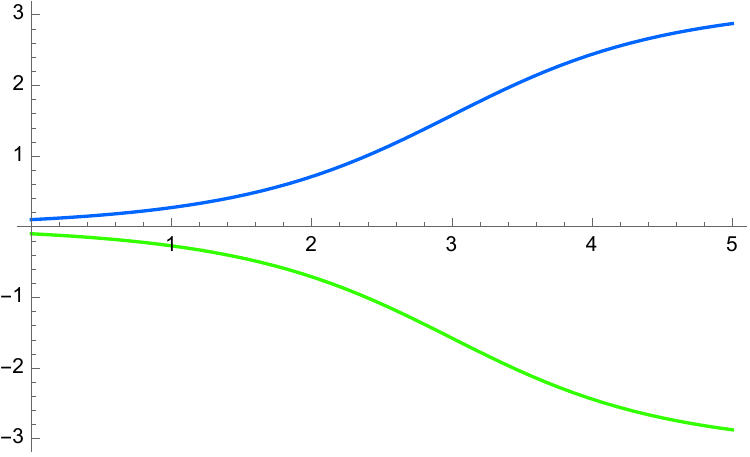}%
\caption{Example C.2: $\phi_1(a)=|\sin(a/2)|$,  $\phi_2(a)=-\phi_1(a)$}
\end{subfigure}
\hfill
\begin{subfigure}{0.48\textwidth}
\includegraphics[width=\textwidth]{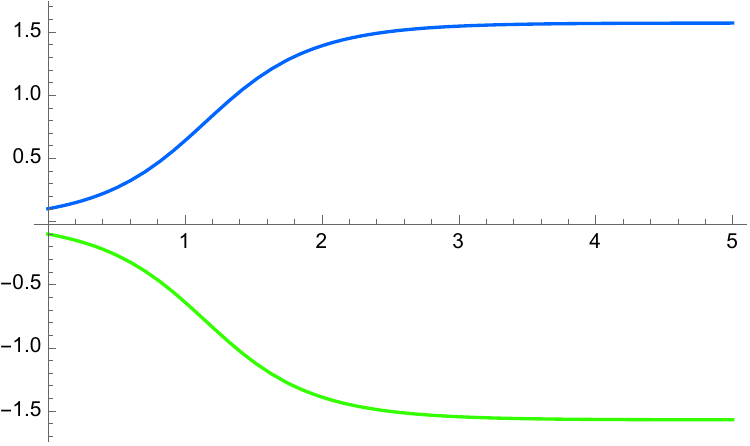}%
\caption{Example C.3: $\phi_1(a)=\sin(a)$,  $\phi_2(a)=-\phi_1(a)$}

\end{subfigure}
\hfill
\begin{subfigure}{0.48\textwidth}
\includegraphics[width=\textwidth]{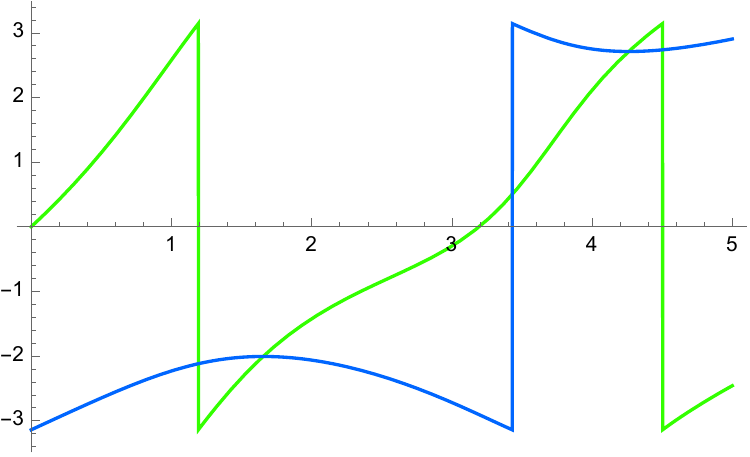}%
\caption{Example C.4: $\phi_1(a)=\sin(2a)+2$,  $\phi_2(a)=\sin(a)$}

\end{subfigure}
\hfill
\caption{Examples C (1 to 4).}
\label{f-S1}
\end{figure}

\subsection{Action of $SO(3,\R)$ on the sphere $S^{2}$}

In this section, we describe population dynamics on the two-dimensional sphere $S^2\subset \R^3$. From now on, we consider a sphere $S^2$ of radius 1. We use polar coordinates on the $(x,y)$ plane and mainly consider $z>0$ for the moment, i.e. we write
$$ x=\rho\cos(\theta),\qquad y=\rho\sin(\theta),\qquad z=\sqrt{1-\rho^2}.$$

We have two different groups that naturally act on the sphere. On one side, we consider the action of $SO(3,\R)$, that represents all rotations of the sphere. In the next Section \ref{s-SO2}, we will discuss the action of $SO(2,\R)$. 

The Lie group $SO(3,\R)$ is the group of all  rotations of $\R^3$ around the zero. Its Lie algebra $so(3,\R)$ is generated by the three following vector fields (written in the $\rho, \theta$-coordinates):%
$$A    =\sqrt{1-\rho^{2}}\left(-\sin(\theta)\partial_\rho-\frac{\cos(\theta)}{\rho}\partial_\theta\right),\qquad
B=\sqrt{1-\rho^{2}}\left(\cos (\theta)\partial_\rho-\frac{\sin(\theta)}{\rho}\partial_\theta\right),\qquad
C   =\partial_\theta.$$
It holds $[A,B]=-C$, $[B,C]=-A$, $[C,A]=-B$.

We now study the FPD in this case.

\begin{proposition} \label{p-SO3} Consider the FPD associated to $(N,S^2,SO(3,\R))$. Then:

\begin{itemize}
\item for $N=1$, the FPD is reduced to the null vector field;
\item for $N=2$, all PDs are of the form 
$$F(X)=\sum_{i=1}^N f_i(X)\partial_{X_i}\qquad\mbox{~~ with ~~}\qquad f_i=\phi_{i,1}(X) h(X_1,X_2)+\phi_{i,2}(X)h^\perp(X_1,X_2),$$
where $h(X_1,X_2)$ is the component of the gradient of the Riemannian distance $d(X_1,X_2)$ on the variable $X_1$ and $h^\perp$ is its orthogonal on $S^2$;

\item for $N\geq 3$, by choosing 3 distinct indexes $k_1,k_2,k_3$, all PDs are of the form 
$$F(X)=\sum_{i=1}^N f_i(X)\partial_{X_i}\qquad\mbox{~~ with ~~}\qquad f_i=\phi_{i,1}(X) h(X_{k_1},X_{k_2})+\phi_{i,2}(X)h(X_{k_1},X_{k_3}),$$
 where $\phi_{i,k}(X)$ are $\hat{\LL}$-jointly radial functions of $X=(X_1,\ldots,X_N)$ and $h(X_a,X_b)$ is the component of the gradient of the Riemannian distance $d(X_a,X_b)$ on the variable $X_a$.
 \end{itemize}
\end{proposition}
\begin{proof} We first consider the case $N=1$. Let $F$ be a PD, whose domain is required to be open, dense and $G$-invariant: since $SO(3)$ acts transitively on $S^2$, its domain needs to be the whole $S^2$. Since $S^2$ has Euler characteristic 2, then $F$ has at least one zero, i.e. $F(\bar X)=0$ for some $\bar X\in S^2$. By equivariance of the vector field $F$, all points in the orbit of $\bar X$ satisfy the same condition, i.e. $F(g\bar X)=0$ for all $g\in SO(3)$. Since the action is transitive, we have $F(X)=0$ for all $X\in S^2$.

We now consider the case $N=2$. We have that $h$ is a solution, by recalling that the distance is radial and applying Proposition \ref{p-GF}. It is easy to prove that the domain of $h$ is $\{X_1\neq \pm X_2\}$, that is $G$-invariant and connected. Then, the definition of $h^\perp$ is unique, up to a single change of sign. It is a solution, due to Corollary \ref{c-perp}. It is easy to prove that it is independent on $h$, thus $h,h^\perp$ provide a basis of solutions for the 2-dimensional manifold. By Theorem \ref{t-FPD}, we have the result.

The case $N\geq 3$ is a direct application of Corollary \ref{c-distance}.

%
%
\end{proof}

 \subsection{Action of $SO(2,\R)$ on the sphere $S^2$} \label{s-SO2}
 
In this section, we consider PDs on $S^2$, where the group action is $SO(2,\R)$, that is the Lie group of all rotations of $\R^3$ around the $z$-axis. In this example, the group of rotations preserves a specific axis: this is particularly interesting for models of swarms moving on the Earth that take into account the magnetic field, i.e. that preserve the Earth's rotational axis passing through poles.

From now on, we consider the stereographic projection $(x,y,z)\mapsto \left(\frac{x}{1-z},\frac{y}{1-z}\right)$, that is a diffeomorphism between the sphere without poles $S^2\setminus\{(0,0,\pm 1)\}$ and the plane without the origin $\R^2\setminus\{0\}$. Moreover, with this change of coordinates, the action of $SO(2,\R)$ on $S^2\setminus\{(0,0,\pm 1)\}$ is identical to the classical one (rotations around the origin) on $\R^2\setminus\{0\}$. We then represent equivariant trajectories on $S^2$ by their stereographic projection, for simplicity of visualization.

\begin{proposition} Consider the FPD associated to $(N,\mathbb{R}^{2}\setminus\{0\},SO(2,\R))$ with $N\geq 1$, that is equivalent to the FPD associated to $(N,S^{2}\setminus\{(0,0,\pm1)\},SO(2,\R))$ by stereographic projection. Write $X_i=(x_i,y_i)=\rho_i(\cos(\theta_i),\sin(\theta_i))$ on $\R^2\setminus\{0\}$ in Cartesian and polar coordinates.

PDs are of the form $F={\displaystyle\sum\limits_{i=1}^{N}}
f_{i,1}\partial_{x_{i}}+f_{i,2}\partial_{y_{i}}$, with
$$
\left(\begin{array}{c}%
f_{i,1} \\
f_{i,2} %
\end{array}
\right)=\left(
\begin{array}
[c]{cc}%
x_{i} & -y_{i}\\
y_{i} & x_{i}%
\end{array}
\right) \left(\begin{array}{c}%
\phi_{i}(\rho_{1},\ldots,\rho_{N},\theta_{1}-\theta
_{2},\ldots,\theta_{1}-\theta_{N}) \\
\psi_{i} (\rho_{1},\ldots,\rho_{N},\theta_{1}-\theta
_{2},\ldots,\theta_{1}-\theta_{N})%
\end{array}
\right).
$$
\end{proposition}
\begin{proof} Recall that the centralizer $\Centr$ of $so(2,\R)$ is $\R.\Id+\R.J$ with $J$ given by \eqref{e-J}. Since $\Centr$ has dimension 2, one can apply Theorem \ref{mainth}, Item 1. The space of solutions is then given by 
$\phi(X) X+\psi(X)JX$ with $\phi(X),\psi(X)$ being invariant with respect to the action of $SO(2,\R)$. By writing in polar coordinates $X_i=\rho_i(\cos(\theta_i),\sin(\theta_i))$, one can apply a rotation of angle $-\theta_1$ and introduce the variables $Y_i=\rho_i(\cos(\theta_i-\theta_1),\sin(\theta_i-\theta_1))$. Remark that, contrarily to the Euclidean case studied in Section \ref{s-SE2}, we still have the non-zero variable $Y_1=\rho_1(1,0)$. This implies that $\phi,\psi$ depend on variables $\rho_1,\ldots,\rho_N, \theta_2-\theta_1,\ldots,\theta_N-\theta_1$. \end{proof}

\begin{remark} \label{r-proper-inclusion-1} For $N=1$, we observe that the space of elementary solutions is strictly contained in the space of solutions, i.e. $\EE\subsetneq \SS$. Indeed, it is reduced to solutions of the form $\phi(X) X$. We are anyway able to completely characterize solutions, since $\CC=\SS$, i.e. solutions given by the centralizer are all solutions.

For $N\geq 2$, one can find alternative (equivalent) expressions for the PDs, by applying Theorem \ref{mainth}, Item 2.
\end{remark}

We now show some examples of PDs associated to $(N,\R^2\setminus\{0\},SO(2,\R))$ in the simple cases of $N=1,2$.

For $N=1$, in Figure \ref{f-S2}-Left, we show four different examples of PDs, all starting from $(x,y)=(1,1)$. We show that $\phi>0$ corresponds to move away from the origin, while $\psi>0$ corresponds to rotate counter-clockwise. Opposite signs mean opposite behavior. In particular, in Examples D.1, D.2 we have positive $\phi$, with zero and negative $\psi$, respectively. In Example D.3, we have $\phi\equiv 0$ and $\psi=10$, providing a periodic circular trajectory with constant angular velocity. In Example D.4, we have $\phi(r)=2-4r$, that provides stabilization of the radius to $0.5$; at the same time, $\psi=50>0$ provides rotation with constant angular velocity. Example D.4, that has a nature similar to Example A.3 above, shows that one can stabilize the system towards a given radius (i.e. a given parallel on $S^2$). Instead, stabilization of the angular variable (or both radial and angular) is not achievable with such PDs, due to the invariance with respect to rotations.


For $N=2$, in Figure \ref{f-S2}-Right we show three different examples of PDs, in which agent 1 is marked in green and agent 2 is marked in blue. Similarly to the previous case, we can have convergence of both the radius and angular variables (Example E.1), eventually choosing the radius (Example E.2, the target radius being $r=2$). One can also converge to a chosen angular difference (Example E.3, angular difference $\pi/8$), but cannot send agents to a chosen common angle. We explicitly have the following choices of functional parameters:

\begin{description}
\item[Example E.1] $\phi_1=\rho_2-\rho_1=-\phi_2$, $\psi_1=\sin(\theta_2-\theta_1)=-\psi_2$.

\item[Example E.2] $\phi_1=2-\rho_1$, $\phi_2=2-\rho_2$, $\psi_1=\sin(\theta_2-\theta_1)=-\psi_2$.

\item[Example E.3] $\phi_1=0.1$, $\phi_2=\rho_1-\rho_2$, $\psi_1=\sin(\theta_2-\theta_1-\pi/8)=-\psi_2$.
\end{description}

\begin{figure}[htb]
\centering
\begin{subfigure}{0.48\textwidth}
\includegraphics[width=\textwidth]{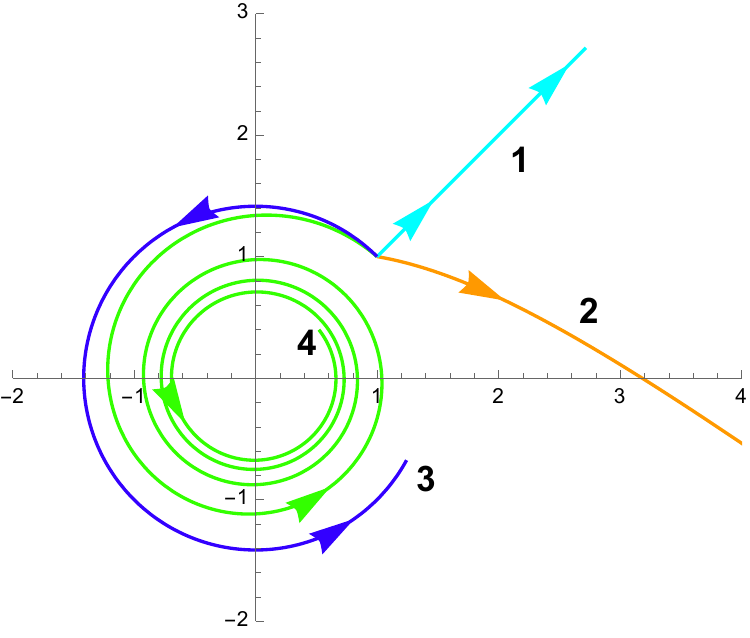}%
\caption{Examples D (1 to 4).}
\end{subfigure}
\hfill
\begin{subfigure}{0.4\textwidth}
\includegraphics[width=\textwidth]{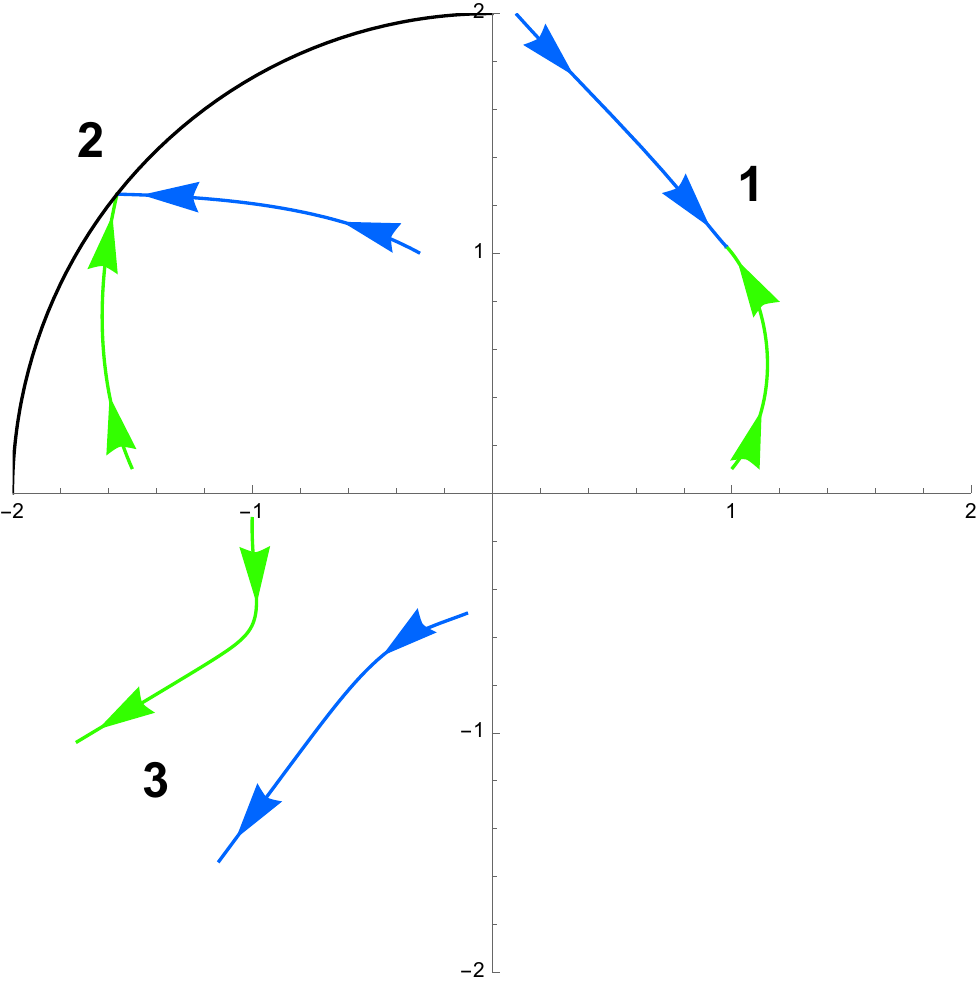}%
\caption{Examples E (1 to 3).}
\end{subfigure}

\caption{Examples on $S^2$ with $N=1$ and $N=2$.}
\label{f-S2}
\end{figure}


%

%
%
%

\section{Volume preserving Population Dynamics on the plane}\label{s-SL2}
  
In this section, we consider Population Dynamics on the punctured plane $\R^2\setminus\{0\}$ under the action of $SL(2)$. The case is interesting, since it is the only one, among our examples, in which the group is not a group of isometries (or pseudo-isometries).

The resulting dynamics on the punctured plane $\R^2\setminus\{0\}$ preserve volumes: given two points $X_1=(x_1,y_1)$ and $X_2=(x_2,y_2)$, we denote by $$X_1\wedge X_2:=\left|\begin{array}{cc}x_1&x_2\\
y_1&y_2\end{array}\right|=x_1y_2-y_1x_2$$
the (oriented) area of the parallelogram with vertices $0,X_1,X_1+X_2,X_2$. It is easy to prove that volumes are preserved by the action of the special linear group $SL(2,\R)$ composed of matrices with determinant 1.

We now describe jointly radial functions and PDs for the action of $SL(2)$.
\begin{proposition} \label{p-sl2} Let $M=\mathbb{R}^{2}\setminus\{0\}$. Consider the set $\mathcal{O}=\{X_1\wedge X_2\neq 0\}\subset M^N$, that it is open dense, $G$-invariant.

The jointly radial functions on $M^N$ defined on $\mathcal{O}$ are 
\begin{equation}\label{e-jr-sl2}
\psi(X_1,\ldots,X_N)=\phi(X_1\wedge X_2, [X_1 | X_2]^{-1} X_3,\ldots, [X_1 | X_2]^{-1} X_N),
\end{equation}
where $[X_1 | X_2]$ is the $2\times 2$-matrix with columns given by $X_1,X_2$ and $\phi$ is an analytic function of its $N-2$ variables in $\R^2$. 
Consider now the FPD associated to $(N,,SL(2,\R))$. 
\begin{itemize}
\item For $N=1$, all PDs are of the form $F=\alpha \left(x_{1}\partial_{ x_{1}}+y_{1}\partial_{ y_{1}}\right)$ for $\alpha\in\mathbb{R}$. 

\item For $N\geq 2$, choose two distinct indexes $j_1,j_2\in\{1,\ldots,N\}$. Then, all PDs are of the form $$F=\sum_{i=1}^N \phi_{1,i}\left(x_{j_1}\partial_{ x_{i}}+y_{j_1}\partial_{ y_{i}}\right)+\phi_{2,i}\left(x_{j_2}\partial_{ x_{i}}+y_{j_2}\partial_{ y_{i}}\right)$$ where $\phi_{i,1},\phi_{i,2} $ are arbitrary jointly-radial functions, i.e. of the form \eqref{e-jr-sl2}.
\end{itemize}
\end{proposition}
\begin{proof} We first prove that jointly radial functions are of the form \eqref{e-jr-sl2}. It is easy to prove that $\psi$ of such form is jointly equivariant: the diagonal action of $g\in SL(2)$ reads as
\begin{eqnarray*}
\psi(gX_1,\ldots,gX_N)&=&\phi((gX_1)\wedge (gX_2), [gX_1 | gX_2]^{-1} (gX_3),\ldots, [gX_1 | gX_2]^{-1} gX_N)\\
&=&\phi(X_1\wedge X_2, [X_1 | X_2]^{-1} X_3,\ldots, [X_1 | X_2]^{-1} X_N)=\psi(X_1,\ldots,X_N).
\end{eqnarray*}
Here we used that the determinant is invariant under the action of $g$ and that 
$$[gX_1 | gX_2]^{-1} gX_i=[X_1 | X_2]^{-1} g^{-1}gX_i=[X_1 | X_2]^{-1} X_i.$$

It is clear that all jointly radial functions are of this form, since the variables in \eqref{e-jr-sl2} play the role of coordinates in the quotient $\mathcal{O}/G$. 

We now study the case $N=1$. Let $F$ be a PD, which domain is a $G$-space. Since $SL(2,\R)$ acts transitively on $M$, its domain is the whole $M$. Write $F(X)=f(x,y)\partial_x+g(x,y)\partial_y$ and write $[F,l]=0$ for $l$ in the basis of the Lie algebra $so(2,\R)$, that is 
$$l_1=y\partial_x,\qquad l_2=x\partial_y,\qquad l_3=x\partial_x-y\partial_y.$$
A direct computation shows:
$$ [F,l_1]=g\partial_x-y(\partial_x f) \partial_x-y(\partial_x g) \partial_y,\qquad [F,l_2]=f\partial_y-x(\partial_y f) \partial_x-x(\partial_y g) \partial_y.$$
This implies that $g=g(y)$ and $f=f(x)$. Since $g(y)=y \partial_x f$, then $g=k y$ for some $k\in \R$. Similarly, $f=kx$ for the same $k$. By linearity, we are now left to prove that  $F(x)=x\partial_x+y\partial_y$ satisfies 
$$[F,l_3]=(x\partial_x-y\partial_y)-(x\partial_x-y\partial_y)=0.$$
This proves the result.

The case $N\geq 2$ is a direct application of Theorem \ref{mainth}, Item 2.
\end{proof}

%

\section{Controlled Population Dynamics: the unicycle} \label{s-unicycle}
In this section, we describe an interesting problem of Controlled Population Dynamics as given in Definition \ref{d-CPD}, i.e. a PD in which the dynamics is constrained to belong to a subset $\Omega$ of the Lie algebra of vector fields. In the terminology of mechanical systems, constraints on the velocity are known as {\it non-holonomic}.

\subsection{The classical unicycle}\label{s-class-uni}
 The unicycle is a classical model of a vehicle on the plane $\R^2$ moving along straight lines or turning on itself, or combining these displacements. The state space is then $\R^2\times S^1$, i.e. the state is described by the position-orientation variables $(x,y,\theta)$. The dynamics is given by 
\begin{equation}\label{unicycle}
\dot x=\cos(\theta)u(t),\qquad \dot y=\sin(\theta)u(t),\qquad \dot \theta=v(t). %
\end{equation}
Here, $u(t),v(t)$ are the control variables, accounting for straight movements and rotations, respectively. Remark that the constraint is linear, since $\Omega$ is generated by $F_1=\cos(\theta)\partial_x+\sin(\theta)\partial_y, F_2=\partial_\theta$. It is not a Lie algebra, since it holds $[F_1,F_2]=-\sin(\theta)\partial_x+\cos(\theta)\partial_y\not\in\Omega$.

One of the interests of this model in the context of PDs is that it is a first-order system that allows to describe orientation of the displacements (via the variable $\theta$). This is in sharp contrast with other models, such as the celebrated Cucker-Smale model \cite{CS}, in which orientation is encoded by a second-order dynamics. Clearly, the  advantage of dealing with first-order systems is here mitigated by the non-holonomic constraint.

We now impose equivariance with respect to the action of the group $G=SE(2)$ of rototranslation of the plane. The action on $\R^2\times S^1$ is the following: given an element 
$$g=\left(\begin{array}{ccc}
\cos(\tau)&-\sin(\tau)&a\\
\sin(\tau)&\cos(\tau)&b\\
0&0&1\end{array}\right)
\in SE(2)$$ and $(x,y,\theta)\in \R^2\times S^1$, it holds
\begin{equation}\label{e-SE2-unicycle}
\Phi(g)\left(\begin{array}{c}
x\\
y\\
\theta\end{array}
\right)= \left(\begin{array}{ccc}
\cos(\tau)&-\sin(\tau)&0\\
\sin(\tau)&\cos(\tau)&0\\
0&0&1\end{array}\right)
\left(\begin{array}{c}
x\\
y\\
\theta
\end{array}
\right)+
\left(\begin{array}{c}
a\\
b\\
\tau
\end{array}
\right).
\end{equation}

The Lie algebra $\LL$ is then generated
by the 3 vector fields $\left\{\partial_x,\partial_y,-y\partial_x+x\partial_y+\partial_\theta\right\}$. 

We identify the unicycle variables $(x,y,\theta)$ with the element $X\in SE(2)$ defined by $$X=\left(\begin{array}{ccc}
\cos(\theta)&-\sin(\theta)&x\\
\sin(\theta)&\cos(\theta)&y\\
0&0&1
\end{array}\right).$$

In this framework, equation \eqref{unicycle} describes left-invariant vector fields, while \eqref{e-SE2-unicycle} is the left translation. See the Appendix for more details.

We are now ready to describe the Controlled FPD for the unicycle. 
\begin{proposition} Consider the Family of CPD associated to $(N,\mathbb{R}^{2}\times S^1,SO(2,\R)\ltimes
\mathbb{R}^{2},\Omega)$ with $N\geq 1$ and $\Omega$ given by \eqref{unicycle}. Denote by $X_i=(x_i,y_i,\theta_i)$ the state of the $i$-th agent and $X:=(X_1,\ldots,X_N)$. Let 
$$\dot X_i=u_i(X)F_i+v_i(X)G_i\qquad\qquad\mbox{ with }\qquad F_i=\cos(\theta_i)\partial_{x_i}+\sin(\theta_i)\partial_{y_i}\quad\mbox{~~ and ~~}\quad G_i=\partial_{\theta_i}.$$

Then, the controlled FPD is characterized by the fact that
$u_{i},v_{i}$ are analytic functions of $\rho_2,\ldots,\rho_N$, $\alpha_2-\theta_1,\ldots,\alpha_N-\theta_1,\theta_2-\theta_1,\ldots,\theta_N-\theta_1$ only, where $Z_i:=(x_i-x_1,y_i-y_1)=\rho_i(\cos(\alpha_i),\sin(\alpha_i))$ for $i=2,\ldots,N$.
\end{proposition}
\begin{proof}  We first study FPD on $SE(2)$ without the constraint \eqref{unicycle}, that are completely characterized by Proposition \ref{p-LieGroups}. Indeed, the FPD is given by
\begin{equation}\label{e-FPD-unicycle}
\dot X_i=\sum_{j=1}^3 \psi_{ij}(X_1,\ldots,X_N)(X_i\cdot l^j),
\end{equation}
where $l^1,l^2,l^3$ is a basis of the Lie algebra, hence $X \cdot l^1,X \cdot l^2, X \cdot l^3$ is a basis of the left-invariant vector fields. Such a basis can be chosen as
\begin{equation}l^1=\left(\begin{array}{ccc}
0&0&1\\
0&0&0\\
0&0&0
\end{array}\right),\qquad l^2=\left(\begin{array}{ccc}
0&-1&0\\
1&0&0\\
0&0&0
\end{array}\right),\qquad l^3=\left(\begin{array}{ccc}
0&0&0\\
0&0&1\\
0&0&0
\end{array}\right).
\end{equation}

Moreover, the $\psi_{ij}(X_1,\ldots,X_N)$ are jointly radial. Again by Proposition \ref{p-LieGroups}, they are characterized by the fact that they are analytic functions of $\rho_2,\ldots,\rho_N$, $\alpha_2-\theta_1,\ldots,\alpha_N-\theta_1,\theta_2-\theta_1,\ldots,\theta_N-\theta_1$ only, where $Z_i:=(x_i-x_1,y_i-y_1)=\rho_i(\cos(\alpha_i),\sin(\alpha_i))$ for $i=2,\ldots,N$.

We now identify the Controlled Population Dynamics, by simply verifying which dynamics for the unicycle \eqref{unicycle} are of the form \eqref{e-FPD-unicycle}. Observe that \eqref{unicycle}  reads as $\dot X=u(t)F_1(X)+v(t)F_2(X)$ and that $F_1(X)=X\cdot l^1$, $F_2(X)=X\cdot l^2$. As a consequence, solutions of the unicycle \eqref{unicycle} are of the form \eqref{e-FPD-unicycle} if and only if $\psi_{i3}\equiv 0$. This proves the result.
\end{proof}

We now discuss examples. For $N=1$, controls $u,v$ need to be constant, hence providing circular trajectories (that we do not present here). We provide a single example for $N=2$ in Figure \ref{f-unicycle}, where $u_1=0.3\rho_2, v_1=-0.2\sin(\theta_1-\alpha_2), u_2=0.2\rho_2, v_2=-0.1 \rho_2$. With such choice, agent 1 (green) aims to reach agent 2 (blue) with the same angle. Moreover, the dynamics promotes convergence to an equilibrium, in which $\rho_2=0$.

 \begin{figure}[htb]%
\centering
\includegraphics[width=7cm]{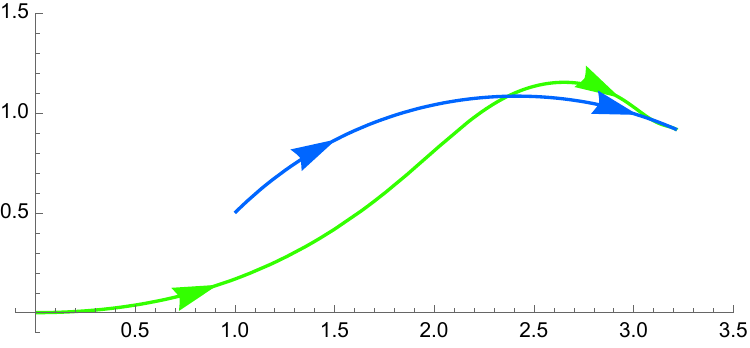}%
\caption{A Controlled Population Dynamics for the unicycle with $N=2$.}%
\label{f-unicycle}%
\end{figure}

%
%
%

\subsection{The relativistic unicycle\label{relatuni}}
In this section, we consider a unicycle in space-time, by adding the time variable $T$. The dynamics then takes place in $\R^3\times S^1$ and it can be written as follows:
\begin{equation}
\dot T=f(t),\qquad \dot x=\cos(\theta)u(t),\qquad \dot y=\sin(\theta)u(t),\qquad \dot \theta=v(t). \label{uni1}
\end{equation}

The natural group of equivariance for the variables $(T,x,y)$ only is the group $SO(1,2)\ltimes \R^3$, that we already studied in Section \ref{s-relat3}. No natural action of such group seems to take into account the non-holonomic angular constraint for $(x,y,\theta)$ in \eqref{uni1}.

A natural group of equivariance is $K=SO(1,1)\times SO(2,\R)$, acting as follows: by introducing polar coordinates for $(x,y)=\rho(\cos(\alpha),\sin(\alpha))$, the action is $(g,\omega).(T,\rho,\alpha,\theta)=(g.(T,\rho),\alpha+\omega,\theta+\omega) $, where $g\in SO(1,1)$ is written as $g=\left(
\begin{array}
[c]{cc}%
\cosh(\lambda) & \sinh(\lambda)\\
\sinh(\lambda) & \cosh(\lambda)
\end{array}
\right)  $. The action can be seen as the natural ``Lorentzian'' group action, as it preserves the metric $ds^{2}=dT^{2}-d\rho^{2} $. The rotation $\omega$ simultaneously acts on the variables $\alpha$ and $\theta$. We have the following result.

\begin{proposition} Consider the controlled FPD associated to the relativistic unicycle $(N,\R^3\times S^1,SO(1,1)\times SO(2,\R))$.

For $N=1$, the controlled FPD is composed by vector fields
\begin{equation}\label{e-rel-uni-1}
T\cos(\theta-\alpha)\bar{u}(r,\theta-\alpha)\partial_T%
+\rho\cos(\theta-\alpha)\bar{u}(r,\theta-\alpha)\partial_\rho
+\sin(\theta-\alpha)\bar{u}(r,\theta-\alpha)\partial_\alpha
+v(r,\theta-\alpha)\partial_\theta
\end{equation}
with $\bar u,v$ analytic functions and $r=\sqrt{T^2-\rho^2}$.

For $N\geq 2$, the FPD {\bf contains} vector fields $F=\sum_{i=1}^N F_i \partial_{X_i}$ with 
\begin{eqnarray}\label{e-rel-uni-N}
F_i&=&T_i\cos(\theta_i-\alpha_i)\bar{u}_i(X_1,\ldots,X_N)\partial_{T_i} 
+\rho_i\cos(\theta_i-\alpha_i)\bar{u}_i(X_1,\ldots,X_N)\partial_{\rho_i}
+\\
&&\sin(\theta_i-\alpha_i)\bar{u}_i(X_1,\ldots,X_N)\partial_{\alpha_i}
+\bar{v}_i(X_1,\ldots,X_N)\partial_{\theta_i},\nonumber
\end{eqnarray}
where $\bar{u}_i,\bar{v}_i$ are jointly radial functions of their variables.
\end{proposition}
\begin{proof} We first focus on $N=1$. We use polar coordinates for $(x,y)=\rho(\cos(\alpha),\sin(\alpha))$. The system reads as 
$$\dot T=f(T,\rho,\alpha,\theta),\qquad \dot \rho=\cos(\theta-\alpha)u(T,\rho,\alpha,\theta),\qquad \dot \alpha=\frac1\rho \sin(\theta-\alpha)u(T,\rho,\alpha,\theta),\qquad \dot \theta=v(T,\rho,\alpha,\theta).$$
\label{e-uni-rel}
The action of $SO(2,\R)$ implies that the dynamics depends on the variables $T,\rho,\theta-\alpha$ only, thus $f,u,v$ only depend on them.

The action of $SO(1,1)$ implies two conditions: first, that functions $f,u,v$ depend on the variables $r,\theta-\alpha$ only, with $r=\sqrt{T^2-\rho^2}$. Second, it implies $f(T,\rho,\theta-\alpha)=\psi(r,\theta-\alpha) T$ and $\cos(\theta-\alpha)u(T,\rho,\theta-\alpha)=\psi(r,\theta-\alpha)\rho$ for an analytic function $\psi(r,\theta-\alpha)$. Computations are identical to the proof of Proposition \ref{p-relat} with $N=2$. We now write $\psi(r,\theta-\alpha)=\cos(\theta-\alpha)\bar u(r,\theta-\alpha)$ for an analytic function $\bar u(r,\theta-\alpha)=\frac{u(t)}{\rho}$ and find \eqref{e-rel-uni-1}.

For $N\geq 2$, it is easy to prove that $F$ defined by \eqref{e-rel-uni-N} are PD, with the same computations.
\end{proof}

\section{Population Dynamics for quantum agents}\label{s-quantum}
In this section, we discuss the problem of joint equivariance of systems of quantum particles. In the standard formalism of quantum theory (see e.g. \cite{hall}), particles are described by (square roots of) probability densities. We consider a finite-dimensional dynamics, i.e. $n\in \mathbb{N}$ possible states. The Schr\"{o}dinger equation describing the evolution of a single particle is then:%
\begin{equation}
\dot{\Psi}=A\Psi,\label{schro}%
\end{equation}
where $A$ is a skew-adjoint operator over $\mathbb{C}^{n}\cong l^{2}(\{1,\ldots,n\},\mathbb{C})$. Here we have $<\Psi,\bar{\Psi}>=1$ and the probability that the particle is in the state $k\in\{1,\ldots,n\}$ is $p_{k}=\Psi_{k}\bar{\Psi}_{k}$.

We start by considering $N$ non-interacting particles, indexed by $i\in\{1,\ldots,N\}$. The particle $i$ has $n_i$ possible states, the quantum state is denoted by $\Psi_i$ and the dynamics is given by a matrix $A_i$.
It is clear that the state of the whole system is $X=\Psi_{1}\otimes\Psi_{2}\otimes\ldots\otimes\Psi_{N}\in\mathbb{C}^{n_{1}}\otimes\ldots\otimes\mathbb{C}^{n_{N}}$, since the probability of $X$ to
be in the state $(k_{1},\ldots,k_{N})$ is just $p_{k_{1}}p_{k_{2}}\ldots p_{k_{N}}$. Its evolution is indeed given by
\begin{eqnarray*}
X(t)  & =&\Psi_{1}(t)\otimes\Psi_{2}(t)\otimes\ldots\otimes\Psi_{N}
(t)\label{evolution}\\
& =&(e^{tA_{1}}\otimes\ldots\otimes e^{tA_{N}})(\Psi_{1}(0)\otimes\Psi_{2}%
(0)\otimes\ldots\otimes\Psi_{N}(0)).
\end{eqnarray*}

By computing the corresponding infinitesimal generator $\mathcal{A}$, we have
\begin{eqnarray*}
\mathcal{A}  & =&A_{1}\otimes \Id_{n_{2}}\otimes\ldots\otimes \Id_{n_{N}} +\Id_{n_1}\otimes A_{2}\otimes\ldots\otimes \Id_{n_{N}} +\Id_{n_{1}}\otimes\ldots\otimes \Id_{n_{N-1}}\otimes A_{N},
\end{eqnarray*}
where $\Id_n$ is the identity matrix of dimension $n$. The key point is then the following: while in classical kinematic systems the joint dynamics occurs on the direct sum of the state spaces, here the natural context is the tensor product of the individual state spaces.

Assume now that all particles lie in the same state space, that is then $\mathbb{C}^n$ for a common dimension $n\in\mathbb{N}$. The state space for the whole system is then $(\mathbb{C}^{n})^{\otimes N}\simeq \mathbb{C}^{(n^N)}$.  Let $G$ be a symmetry group acting on $\mathbb{C}^n$, and denote by $\LL$ its Lie algebra. Then, the infinitesimal joint action on $(\mathbb{C}^{n})^{\otimes N}$ for $l \in\LL$ is
\begin{eqnarray*}
\mathcal{A}  &=&l\otimes \Id_{n}\otimes\ldots\otimes \Id_{n} +\Id_{n}\otimes l\otimes\ldots\otimes \Id_{n}  +\Id_{n}\otimes\ldots\otimes \Id_{n}\otimes l.
\end{eqnarray*}

This general idea is now developed on a specific case.

\subsection{Family of Population Dynamics for two quantum agents in $\C^2$} 

In this section, we present a simple case for the general theory of PDs with quantum symmetries. We focus on FPD of $N=2$ agents with state $\Psi=\Psi_1\otimes \Psi_2$ evolving on $\C^2\otimes \C^2$ under the $SU(2)$ symmetry. This is one of the most important symmetries in quantum physics, the symmetry of quantum spins. We require the dynamics to be equivariant with respect to its diagonal action, i.e. with respect to the action 
\begin{equation}
\label{e-action-SU2}
g.(\Psi_1\otimes\Psi_2)=g\Psi_1\otimes g\Psi_2.
\end{equation}

The Lie algebra is then $\LL=su(2)$. We want to characterize vector fields $F$ satisfying
\begin{equation}
\lbrack F(\Psi),(l\otimes \Id_{2})\Psi+(\Id_{2}\otimes l)\Psi]=0\qquad \mbox{for all ~~~}%
l\in su(2).\label{maineqqQuant}%
\end{equation}

From now on, we use the following identification: $\C^2\otimes \C^2=\mathcal{M}_{2}(\mathbb{C})$, i.e. the space of quadratic forms on $\C$. We first observe that the action \eqref{e-action-SU2} now reads as 
\begin{equation}
\label{e-action-M2}
g.X=gXg^T,
\end{equation} where $X\in \mathcal{M}_2(\C)$ and $g^T$ is the transpose of $g\in SU(2)$. Remark that the action \eqref{e-action-M2} is reducible. Indeed, by using the Clebsch-Gordan decomposition \cite{hall}, the unitary {\bf diagonal} action of $SU(2)$ over $\mathcal{M}_{2}(\mathbb{C})$ is (unitarily equivalent to) the direct sum of two irreducible components: the (natural) action of $SU(2)$ on symmetric and skew-symmetric matrices, respectively. This is the \textbf{4=3+1} of physicists.
\begin{remark} We stress here that we deal with (standard) transposition of matrices, and not on the conjugate transposition. This explains the interest of the Takagi decomposition used below.
\end{remark}

The structure of skew-symmetric matrices in $\mathcal{M}_{2}(\mathbb{C})$ is very simple, as they satisfy $S=a J$ for some $a\in\C$ and $J$ is given by \eqref{e-J}. For symmetric matrices, we need the following useful decomposition.

 \begin{proposition}[Takagi decomposition \cite{takagi}] Any complex symmetric matrix $S$ of dimension $2\times 2$ satisfies the following decomposition: there exists $U\in U(2)$ unitary matrix such that 
\[
S=U\Delta U^T\qquad \mbox{~~with~~}\qquad\Delta=\left(
\begin{array}
[c]{cc}%
\delta_{1} & 0\\
0 & \delta_{2}%
\end{array}
\right).
\]
Here, $\delta_{1},\delta_{2}\geq0$ are the singular values of $S$.

As a consequence, there exists $\omega$ phase factor and $H\in SU(2)$ such that 
\begin{equation}
S=e^{i\omega}H\Delta H^T.\label{e-takagi}
\end{equation}
Moreover, for $\delta_{1}\neq\delta_{2}$, the parameters $(\Delta,\omega,H)$ are uniquely determined by $S$.

As a further consequence, define $\mathcal{D}_{S}$ the set of complex symmetric
matrices with distinct singular values, that is an open dense $SU(2)$-space, i.e. invariant under the action of $SU(2)$. For all $S\in\mathcal{D}_S$, the Takagi invariants $(\Delta,\omega)$ form a complete set of invariants under the action \eqref{e-action-M2}. Moreover, these invariants smoothly depend on
$S$.
\end{proposition}
\begin{proof} See \cite{takagi}. It is clear that $\delta_1,\delta_2$ are the singular values, since $$S^* S=(U^T)^* \Delta U^* U \Delta U^T=(U^T)^* \Delta^2 U^T.$$

The passage from $U\in U(2)$ to $H\in SU(2)$ is also direct, since any unitary matrix satisfies $U=\exp(i\alpha) H$ with $\alpha\in \R/(2\pi)$. Thus, equation \eqref{e-takagi} holds with $\omega=2\alpha$.
\end{proof}

We are now ready to describe jointly radial functions. The idea here is to use the (unique) Takagi decomposition of the symmetric part of $X$ to define invariants.

\begin{proposition} Let $\mathcal{D}$ be the set of matrices $X$ in $\mathcal{M}_{2}(\mathbb{C})$ such that the symmetric part $S$ of $X$ has distinct singular values. The following holds:
\begin{itemize}
\item the set $\mathcal{D}$ is a $G$-space, i.e. open, dense and invariant under the diagonal action of $SU(2)$;
\item by writing $X=S+A$ and applying the Takagi decomposition \eqref{e-takagi} to $S$, it holds that the set of 5 real parameters $(\omega(S),\delta_{1}(S),\delta_{2}(S),A)$ is a complete set of independent invariants.
\end{itemize}
As a consequence, the ring $\mathcal{R}$ of jointly radial functions with domain $\mathcal{D}$ is the space of functions of such parameters.
\end{proposition}
\begin{proof} The first result is a direct consequence of the fact that the diagonal action of $SU(2)$ on $\mathcal{M}_{2}(\mathbb{C})$ is decomposed into the actions on the symmetric and skew-symmetric part. Thus, if $S$ has distinct singular values, it belongs to $\mathcal{D}_S$, that is $G$-invariant.

As for the second result, it is clear that $\omega(S),\delta_{1}(S),\delta_{2}(S)$ are invariants, as a consequence of the Takagi decomposition. Moreover, the matrix $H\in SU(2)$ is unique, thus we can apply it to the matrix $X=S+A$, thus in particular to its skew-symmetric part $A=a J$. A direct computation shows that it holds $H A H^T=A$, thus the  parameter $a\in \C$ is invariant under the group action. By considering the dimension over the space $\R$, we then have the 5-dimensional set of parameters $(\omega(S),\delta_{1}(S),\delta_{2}(S),a)$ for the 8-dimensional space $\mathcal{M}_{2}(\mathbb{C})$, under the action of the 3-dimensional group $SU(2)$. Thus, the set of parameters is complete.
\end{proof}

We are now ready to describe some PDs for $(N=2,\C^2\otimes \C^2,SU(2))$. Due to the identification $\C^2\otimes \C^2=\mathcal{M}_{2}(\mathbb{C})$, we need to rewrite the Lie algebra condition \eqref{maineqqQuant}. By computing the infinitesimal generator of \eqref{e-action-M2}, it reads as $l.X+X.l^T$ for some $l\in su(2)$, thus condition \eqref{maineqqQuant} reads as 
\begin{equation}
[F(X),l.X+X.l^T]=0\qquad \mbox{for all ~~~} l\in su(2).\label{e-su2}%
\end{equation}
We then aim to describe all vector fields $F$ satisfying condition \eqref{e-su2}. It is clear that the dynamics on $\mathcal{M}_{2}(\mathbb{C})$ can also be decomposed in the symmetric and skew-symmetric parts. Thus, the infinitesimal generator of the group action $(l.X+X.l^T)\partial_X$ now reads as $(l.S+S.l^{T})\partial_S+(l.A-A.l^{T})\partial_A$. A general vector field now reads as $F(S,A)=f_{1}(S,A)\partial_S+f_{2}(S,A)\partial_A$. Condition \eqref{e-su2} splits as follows: 
\begin{eqnarray}
&&f_1 (lS+Sl^T)-(lS+Sl^T)\partial_Sf_1-(lA-Al^T)\partial_Af_1=0\label{e-S}\\
&&f_2 (lA+Al^T)-(lS+Sl^T)\partial_Sf_2-(lA-Al^T)\partial_Af_2=0\label{e-A}
\end{eqnarray}
These equations play the same role as \eqref{maineq1} for linear actions on Euclidean spaces. Even though the equations are different, they share a common important feature: solutions of \eqref{e-S} are a module over the ring $\mathcal{R}$ of jointly radial functions, and the same holds for solutions of \eqref{e-A}. The proof is identical to Theorem \ref{t-FPD}.

We are unable to provide all solutions to \eqref{e-S}, while completely solving \eqref{e-A} is easy. We are interested in the dynamics preserving the unit ball of $\C^2\otimes\C^2$. Our result is stated here.
\begin{proposition} \label{p-quantum} Consider the FPD associated to  $(N=2,\C^2\otimes \C^2,SU(2))$, i.e. to  $(N=2,\mathcal{M}_2(\C),SU(2))$ preserving the unit ball of $\C^2\otimes\C^2$. Decompose $X\in\mathcal{M}_2(\C)$ as $X=S+A$, its symmetric and skew-symmetric part, respectively. Then, the FPD {\it contains} the following vector fields:
$$F(X)= i\phi_{1}(S,A) S\partial_S+i\phi_{2}(S,A) A \partial_A$$
 for any real-valued jointly-radial functions $\phi_1,\phi_2\in \mathcal{R}$. Moreover, all solutions $f_2$ of \eqref{e-A} are of the form $i\phi_2(S,A) A$.

\end{proposition}
\begin{proof} The first result is a direct consequence of the fact that $iS$ is a solution to \eqref{e-S}, while $iA$ is a solution to \eqref{e-A}.

It is clear that solutions of \eqref{e-A} define vector fields on the manifold of skew-symmetric matrices, that is 2-dimensional at each point. Since the space of solutions of $i\phi_2(S,A) A$ spans the tangent space at each point, we have all solutions.

Preservation of the unit ball is standard, since vector fields now read as $i\phi_1 S\partial_S+i\phi_2 A \partial_A$ with $\phi_1,\phi_2$ real functions of variables $S,A$.
\end{proof}
\begin{remark} The (real) dimension of the unit sphere in the space $\C^2\otimes \C^2$ is 7. The space of vector fields given in Proposition \ref{p-quantum} is 2-dimensional (1 for the symmetric and 1 for the skew-symmetric ones). Then, we are missing a 5-dimensional space of vector fields, all related to the dynamics of the symmetric parts.
%
\end{remark}

\appendix

\section*{Appendix: Lie groups and Lie algebras}
\label{a-Lie}

In this appendix, we recall basic definitions and properties of Lie groups and algebras. For a complete treatment, see e.g. \cite{barut}. For simplicity of treatment, we only consider Lie algebras over the field $\R$.

\begin{definition}\label{d-centralizer} A (real) Lie algebra $(\LL,[.,.])$ is a (real) vector space, endowed with a Lie bracket operation $[.,.]:\LL\times\LL\to\LL$, that satisfies bilinearity, anticommutativity ad the Jacobi identity:
$$[[x,y],z]+[[y,z],x]+[[z,x],y]=0\mbox{~~for all~~}x,y,z\in\LL.$$

Given $l\in\LL$, the adjoint map $ad(l):\LL\to\LL$ is given by $ad(l)x:=[l,x]$.

The center of $\LL$ is the set of elements $c\in\LL$ such that $[l,c]=0$ for all $l\in\LL$. Given $X\subset \LL$, the centralizer of $X$ is the set $\Centr$ of elements that commute with elements of $X$, i.e. $$\Centr:=\{c\in \LL\mbox{~~} [x,c]=0\mbox{~~for all~~}x\in X\}.$$
By the Jacobi identity, the centralizer is always a Lie subalgebra of $\LL$.
\end{definition}

We now define Lie groups and related definitions.
\begin{definition} A Lie group $(G,\cdot)$ meets:
\begin{itemize}
\item $(G,\cdot)$ is a group
\item $G$ is a manifold
\item The operation function $(x,y)\mapsto x\cdot y^{-1}$ is a smooth function.
\end{itemize}
The Lie group is analytic if it is an analytic manifold and the operation is analytic. It is connected if it is connected as a manifold. It is linear if it is isomorphic to a subgroup of the group of matrices $(GL(n),\cdot)$ for some $n\in\mathbb{N}$. 
\end{definition}
A relevant case of Lie group is given by semi-direct products. They are fundamental in our work, in particular in Section \ref{s-translation}.
\begin{definition}[Semi-direct product] Let $(G,\cdot)$ be a Lie group acting linearly on a vector space $V$ via the action $\Phi$. The semi-direct product $G\ltimes_\Phi V$ is the Lie group $(G\times V, \ast)$ with the group operation $$(g_1,v_1)\ast(g_2,v_2):=(g_1\cdot g_2, \Phi(g_1)v_2+v_1).$$
\end{definition}

The following fundamental theorem allows to consider real Lie algebras that are given by algebras of matrices only.
\begin{theorem}[Ado's Theorem] Every finite-dimensional Lie algebra $\LL$ over a field of characteristic zero (such as $\R$) is isomorphic to a Lie algebra of square matrices under the commutator bracket $[x,y]:=xy-yx$.
\end{theorem}

We are now ready to recall the exponential of Lie algebras.
\begin{definition} Given a Lie algebra of matrices, we denote with $\exp:\LL\to GL(n)$ the standard matrix exponential. The image of $\LL$ generates a linear Lie group, that is called the exponential of $\LL$. A Lie group is exponential if the exponential map from its Lie algebra is surjective.
\end{definition}
\begin{proposition} Compact connected Lie groups and semi-direct products of compact connected Lie groups by vector spaces are exponential Lie groups.
\end{proposition}

We now recall that the space of vector fields on a manifold $M$, endowed with the classical Lie bracket, is a Lie algebra.
\begin{proposition} Given a smooth manifold $M$, the space of densely-defined vector fields on $M$ is a Lie algebra, with respect to the Lie bracket of vector fields: $[X,Y]=L_X Y$, being the Lie derivative of $Y$ along $X$. One has 
$$[X,Y](x_0):=\frac{d^2}{dt^2}_{|_{t=0}} \exp(-t Y)\circ \exp(-t X) \circ \exp(t Y)\circ \exp(t X)(x_0),$$
where $\exp(t X) (x_0)$ is the flow at time $t$ of the vector field $X$ starting from $x_0$.
\end{proposition}

Given a Lie group $G$, vector fields that are invariant with respect to multiplication play a crucial role, e.g. in application to mechanics. We used them in Section \ref{s-unicycle}.
\begin{definition}[Left- and right-invariant vector fields] Let $(G,\cdot)$ be a Lie group, and $X$ a vector field on it. We say that $X$ is left-invariant if, for any integral curve $\exp(tX)(g_0)$ and $h\in G$, the left translation of the curve is an integral curve too, i.e. $\exp(tX)(h \cdot g_0)=h\cdot \exp(tX)(g_0)$. The definition for right-invariant vector fields is similar.
\end{definition}
\begin{proposition}
For linear Lie groups and algebras, left-invariant vector fields are of the form $X(g)=g\cdot l$ for any $l\in\LL$. Similarly, right-invariant vector fields are of the form $X(g)=l\cdot g$ for any $l\in\LL$.
\end{proposition}

For general Lie groups, the following important results hold.
\begin{proposition} \label{p-LR} Both spaces of left-invariant and right-invariant vector fields on $(G,\cdot)$ form a Lie algebra that is isomorphic to the Lie algebra of $G$.

Given $X$ left-invariant and $Y$ right-invariant vector fields, it holds $[X,Y]=0$, i.e. they commute.
\end{proposition}
\begin{proof} See e.g. \cite[Ch. 3]{barut}. The proof of the last statement is very easy for linear Lie groups: one has $X=g\cdot l_1$ and $Y=l_2\cdot g$ for some $l_1,\l_2\in\LL$. It then holds
$$[X,Y](g)=[g\cdot l_1,l_2\cdot g]=(l_2\cdot g)\cdot l_1 - l_2\cdot (g\cdot l_1)=0.$$
\end{proof}

\end{document}